\documentclass[a4paper]{article}

\usepackage[margin=2.8cm]{geometry}

\usepackage{lmodern}
\usepackage{amsmath,amssymb,amsthm,upgreek,dsfont}

\usepackage{enumitem}
\usepackage[english]{babel}
\usepackage[babel]{csquotes}

\usepackage{xcolor,graphicx,tikz}
\usepackage{ifthen}
\usepackage{cite}

\usepackage[bookmarks=true,bookmarksnumbered,plainpages=false,linktocpage,colorlinks=true,citecolor=green!80!black,linkcolor=red!70!black,filecolor=magenta,urlcolor=magenta,hidelinks,breaklinks,pdfauthor={Paul Geuchen, Thomas Jahn, Hannes Matt},pdftitle={Universal approximation with complex-valued deep narrow neural networks},unicode=true]{hyperref}
\usepackage[nameinlink,sort&compress,capitalize]{cleveref}
\usepackage{cite}
\usepackage{float,subfigure}
\usepackage{aligned-overset}
\usepackage{todonotes}
\usepackage{comment}

\usetikzlibrary{decorations.fractals,calc,intersections,through,backgrounds,arrows,patterns,shapes.geometric,shapes.misc,spy,fadings,decorations.pathreplacing}

\AtBeginDocument{%
   \def\MR#1{}
}

\numberwithin{equation}{section}

\theoremstyle{plain}
\newtheorem{Satz}{Theorem}[section]

\newtheorem{Lem}[Satz]{Lemma}
\newtheorem{Prop}[Satz]{Proposition}
\newtheorem{Assum}[Satz]{Assumption}

\theoremstyle{definition}
\newtheorem{Def}[Satz]{Definition}
\newtheorem{Bsp}[Satz]{Example}
\newtheorem{Bem}[Satz]{Remark}

\crefname{Satz}{Theorem}{Theorems}
\crefname{Prop}{Proposition}{Propositions}
\crefname{Lem}{Lemma}{Lemmas}
\crefname{Kor}{Corollary}{Corollaries}
\crefname{Bem}{Remark}{Remarks}
\crefname{Bsp}{Example}{Examples}
\crefname{Def}{Definition}{Definitions}
\crefname{Assum}{Assumption}{Assumptions}

\newcommand{\ii}{\mathrm{i}}
\newcommand{\ee}{\mathrm{e}}
\newcommand{\dd}{\mathrm{d}}
\newcommand{\NN}{\mathbb{N}}
\newcommand{\RR}{\mathbb{R}}

\newcommand{\CC}{\mathbb{C}}
\newcommand{\mul}{\mathrm{mul}}

\newcommand{\fall}{\:\forall\:}

\newcommand{\bc}[2]{\overline{B}_{#2}(#1)}
\newcommand{\abs}[1]{\left\lvert#1\right\rvert}

\newcommand{\mnorm}[1]{\left\lVert#1\right\rVert}

\newcommand{\setn}[1]{\left\{#1\right\}}
\newcommand{\setns}[1]{\{#1\}}
\newcommand{\setcond}[2]{\left\{#1 \::\: #2\right\}}
\newcommand{\lr}[1]{\!\left(#1\right)}

\newcommand{\defeq}{\mathrel{\mathop:}=}
\newcommand{\eqdef}{=\mathrel{\mathop:}}
\newcommand{\norel}{\mathrel{\phantom{=}}}
\newcommand{\rr}{\varrho}
\newcommand{\pM}[1]{ \begin{pmatrix} #1  \end{pmatrix} }

\newcommand{\wirt}{\partial_{\mathrm{wirt}}}
\newcommand{\wirtq}{\overline{\partial}_{\mathrm{wirt}}}
\newcommand{\p}{\partial}
\DeclareMathOperator{\card}{card}
\DeclareMathOperator{\modrelu}{modReLU}
\DeclareMathOperator{\id}{id}
\DeclareMathOperator{\IM}{Im}
\DeclareMathOperator{\RE}{Re}
\DeclareMathOperator{\Aff}{Aff}
\newcommand{\m}{\textbf{m}}
\newcommand{\elll}{\boldsymbol{\ell}}
\newcommand{\kk}{\textbf{k}}

\let \eps \varepsilon
\let \piup \uppi
\let\emptyset\varnothing

\makeatletter
\renewcommand*{\eqref}[1]{\hyperref[{#1}]{\textup{\tagform@{\ref*{#1}}}}}
\makeatother

\newcommand{\footremember}[2]{\footnote{#2}
    \newcounter{#1}
    \setcounter{#1}{\value{footnote}}}
\newcommand{\footrecall}[1]{\footnotemark[\value{#1}]}

\makeatletter
\newcommand{\subjclass}[2][2020]{\let\@oldtitle\@title \gdef\@title{\@oldtitle\footnotetext{#1 \emph{Mathematics subject classification.} #2}}}
\newcommand{\keywords}[1]{\let\@@oldtitle\@title \gdef\@title{\@@oldtitle\footnotetext{\emph{Key words and phrases.} #1.}}}
\makeatother

\begin{document}

\title{Universal approximation with complex-valued deep narrow neural networks}
\author{Paul Geuchen\footremember{ku}{
Mathematical Institute for Machine Learning and Data Science (MIDS), Catholic University of Eichstätt--Ingolstadt (KU), Auf der Schanz 49, 85049 Ingolstadt, Germany.
Email: \href{mailto:paul.geuchen@ku.de}{\texttt{paul.geuchen@ku.de}}, \href{mailto:thomas.jahn@ku.de}{\texttt{thomas.jahn@ku.de}}, \href{mailto:hannes.matt@ku.de}{\texttt{hannes.matt@ku.de}.\\ All authors contributed equally to this work.}
} \and Thomas Jahn\footrecall{ku} \and Hannes Matt\footrecall{ku}}

\keywords{complex-valued neural networks, holomorphic function, polyharmonic function, uniform approximation, universality}
\subjclass{68T07,41A30,41A63,31A30,30E10}

\maketitle

\begin{abstract}
We study the universality of complex-valued neural networks with bounded widths and arbitrary depths.
Under mild assumptions, we give a full description of those activation functions $\rr:\CC \to \CC$ that have the property that their associated networks are universal, i.e., are capable of approximating continuous functions to arbitrary accuracy on compact domains.
Precisely, we show that deep narrow complex-valued networks are universal if and only if their activation function is neither holomorphic, nor antiholomorphic, nor $\RR$-affine.
This is a much larger class of functions than in the dual setting of arbitrary width and fixed depth.
Unlike in the real case, the sufficient width differs significantly depending on the considered activation function.
We show that a width of $2n+2m+5$ is always sufficient and that in general a width of $\max\setn{2n,2m}$ is necessary.
We prove, however, that a width of $n+m+3$ suffices for a rich subclass of the admissible activation functions.
Here, $n$ and $m$ denote the input and output dimensions of the considered networks.
Moreover, for the case of smooth and non-polyharmonic activation functions, we provide a quantitative approximation bound in terms of the depth of the considered networks.
\end{abstract}

\section{Introduction}

This paper addresses the universality of deep narrow complex-valued neural networks (CVNNs), i.e., the density of neural networks with arbitrarily large depths but bounded widths, in spaces of continuous functions over compact domains with respect to the uniform norm.
Our main theorem is as follows.

\begin{Satz}\label{thm:intro}
Let $n,m\in\NN$, and $\rr:\CC\to\CC$ be a continuous function which at some point is real differentiable with non-vanishing derivative.
Then $\mathcal{NN}^\rr_{n,m,2n+2m+5}$ is universal if and only if $\rr$ is neither holomorphic, nor antiholomorphic, nor $\RR$-affine.
\end{Satz}
Here $\mathcal{NN}^\rr_{n,m,W}$ denotes the set of complex-valued neural networks with input dimension $n$, output dimension $m$, activation function $\rr$, and $W$ neurons per hidden layer.
These neural networks are alternating compositions
\begin{equation}
V_L\circ \rr^{\times W}\circ \ldots \circ \rr^{\times W}\circ V_0: \quad \CC^n\to\CC^m
\end{equation}
of affine maps $V_0:\CC^n\to\CC^W$, $V_1,\ldots,V_{L-1}:\CC^W\to\CC^W$, $V_L:\CC^W\to\CC^m$,
 and componentwise applications of the activation function $\rr$, see \cref{chap:preliminaries} for a detailed definition.

Studying the expressivity of neural networks is an important part of the mathematical analysis of deep learning.
\cref{thm:intro} is a \emph{qualitative} result in that direction.
Such qualitative results naturally precede the investigation of approximation rates, i.e., the decay of approximation errors as the class of approximants increases.
Our focus is on qualitative results, but to show how our methods can also be used to derive quantitative bounds, 
we prove such a result for the case of smooth and non-polyharmonic activation functions, i.e.,
we derive an upper bound on the depth necessary to achieve an approximation accuracy less than $\eps$,
for a prescribed approximation accuracy $\eps>0$; see \cref{thm:quant}.

\subsection{Complex-valued neural networks}\label{sec:cvnn}
Although mostly real-valued neural networks (RVNNs) are used in the field of Deep Learning, recent years have shown a growing interest in the use of complex-valued neural networks in 
various application areas \cite{zhang2021optical,lee2022complex,lei2023fully,qu2023entanglement,arjovsky2016unitary}, 
for instance Magnetic Resonance Imaging (MRI) \cite{virtue2017better,kustner2020cinenet,cole2021analysis} and 
Polarimetric Synthetic Aperture Radar (PolSAR) Imaging \cite{zhang2017complex,ren2023new,barrachina2023comparison}.
These application areas are usually characterized by the fact that complex numbers naturally occur as inputs for machine learning models. 
In such areas, CVNNs are, in contrast to real-valued neural networks, able to handle the complex-valued nature of the inputs in a faithful way, for instance by using a phase-preserving activation function.
Note that this behavior cannot be achieved if one applies a non-trivial real-valued activation function to real and imaginary part of an input separately.

The identification
$\RR^2 \cong \CC$ might suggest that most theoretical properties of CVNNs can directly be derived by those of RVNNs.
However, the two network classes differ in the following two deciding aspects: 
\begin{itemize}
\item The activation function in a CVNN is a function $\varrho: \CC \to \CC$ (i.e., $\RR^2 \to \RR^2$), whereas the activation function in an RVNN is a function $\RR \to \RR$.
Hence, regarding the activation function, CVNNs are more \emph{versatile} than RVNNs. 
\item The affine maps in a CVNN are required to be $\CC$-affine, whereas the affine maps in an RVNN only need to be $\RR$-affine.
Therefore, regarding the affine maps, CVNNs are more \emph{restrictive} than RVNNs. 
\end{itemize}
The observation that CVNNs are on the one hand more versatile and on the other hand more restrictive than RVNNs shows that it is \emph{not}
 possible to obtain theoretical properties of 
CVNNs as a special case of those of RVNNs or vice versa. 
In fact, studying the universality of neural networks of fixed depth
 \cite{Pinkus1999,Voigtlaender2022} has already uncovered significant differences between RVNNs and CVNNs; see also \Cref{sec:contri}.

\subsection{Related work}\label{subsec:related}

In the neural network context, universal approximation theorems date back to the 1980s and 1990s \cite{Cybenko1989,Pinkus1999}, where it was shown that real-valued shallow neural networks with output dimension $1$ and a fixed continuous activation function are universal if and only if the activation function is not a polynomial.
Modifications of the setting in which universal approximation is studied appear in the neural network literature over the past decades.
These variants of the problem refer to, e.g., the input and the output dimension, the target space (typically $L_p$ for $1\leq p\leq \infty$, 
continuous functions, also modulo the action of a group), the choice of activation functions (only ReLU vs.\ any continuous non-polynomial function), 
constraints on either the width or depth of the neural network (narrow vs.\ wide, shallow vs.\ deep networks), or constraints on the norm or the sparsity of the weights, 
see for instance \cite{KidgerLy2020,Cai2022,ParkYuLeSh2020,LuPuWaHuWa2017,Yarotsky2022,MurataSo2017,IsmailovSa2023} and the references therein.
Furthermore, changes in the network architecture \cite{PetersenVo2020,Zhou2020}, the incorporation of randomness \cite{MerkhMo2019}, and changes 
in the nature of the inputs \cite{LiWa2011,Bueno2021} are also subjects of investigation in the literature on universal approximation.

Moreover, the literature contains numerous quantitative statements about the approximation properties of neural networks 
(see for instance \cite{barron1993universal,mhaskar1996neural,yarotsky2018optimal,petersen2018optimal}
and the references therein).
We explicitly mention the paper \cite{kratsios2022universal}, which provides quantitative approximation bounds for deep narrow RVNNs in terms of the depth
of the considered networks (see \cite[Proposition~53]{kratsios2022universal}).
In the present work, we prove a statement similar to \cite[Proposition~53(i)]{kratsios2022universal}
 for CVNNs with a smooth and non-polyharmonic activation function (see \cref{thm:quant}).

Remarkably, the theory mostly covers real-valued neural networks.
Yet, the fact that CVNNs are applied successfully in various application areas (see \cref{sec:cvnn})
motivates the theoretical study of their approximative capabilities.
To the best of our knowledge, the only qualitative and quantitative results in that direction are \cite{Voigtlaender2022,Park2022}, and \cite{CarageaLeMaPfVo2022,Geuchen2023}, respectively.
The article \cite{Voigtlaender2022} provides a characterization of activation functions, for which shallow CVNNs are universal.
This characterization is crucial for the purposes of the paper at hand and is therefore given in \cref{thm:felix} below.
For both the shallow and the deep narrow (real-valued and complex-valued) neural networks with analytic activation functions, it is shown in \cite{Park2022} that the closures of these classes in $C(K)$ coincides with the closure of polynomials, where $K$ is a compact subset of $\RR^n$ or $\CC^n$.
An application of the Stone--Weierstrass theorem or Mergelyan's theorem then yields universality of neural networks in the set of continuous functions in the real-valued case or holomorphic functions in the complex-valued case.
The authors of \cite{CarageaLeMaPfVo2022} prove quantitative bounds for the approximation of $C^k$-functions on $\CC^n$ using complex-valued neural networks with the modReLU activation function.
Those results have recently been generalized in \cite{Geuchen2023}, where the same approximation bounds have been proven for the rich class of complex-valued activation functions that are smooth and non-polyharmonic on some non-empty open set.
This class in particular includes the modReLU.

\subsection{Contribution}\label{sec:contri}

In the real-valued case,  under mild assumptions on their regularity, activation functions that yield universal neural networks have been characterized in the literature.
In the complex-valued case, however, such a characterization is only known in the case of neural networks with fixed depths and arbitrary widths.
To complete the picture, we give in \cref{thm:intro} a characterization of activation functions for which CVNNs with bounded widths and arbitrary depths are universal.

Recall that polynomial activation functions are precisely the ones for which real-valued neural networks with arbitrary widths and fixed depth are not universal, as 
shown by Kidger and Lyons in \cite[Section~1]{KidgerLy2020}.
Yet, in the dual situation, where depth is arbitrary and widths are bounded, polynomial activation functions (of minimum degree $2$) do give rise to universal real-valued neural networks, see again \cite[Theorem~3.2]{KidgerLy2020}.
The situation is different in the complex-valued case.
For continuous activation functions, Voigtlaender shows in \cite[Theorem~1.3]{Voigtlaender2022} that shallow CVNNs are universal if and only if the activation function is non-polyharmonic.
CVNNs of arbitrary widths and fixed depth $>1$ are universal if and only if the activation function does not coincide with a polynomial in $z$ and $\overline{z}$, and is neither holomorphic nor antiholomorphic, cf.~\cite[Theorem~1.4]{Voigtlaender2022}.
In the deep narrow regime studied in \cref{thm:intro}, the requirements on the activation function for universality are again weaker in the sense that not being a polynomial in $z$ and $\overline{z}$ is replaced by not being $\RR$-affine.

In this work, we consider continuous complex-valued activation functions which have non-vanishing derivative (in the sense of real variables) at some point.
Our analysis roughly splits into two parts: polyharmonic and non-polyharmonic activation functions.
While this distinction mainly impacts the proof techniques, the bounds on the widths of the CVNNs are actually governed by the properties of the Wirtinger derivatives of the activation function at the point of differentiability. 
\cref{fig:nutshell} is a graphical guide through our results.

For non-polyharmonic activation functions, we show universality of CVNNs with input dimension $n$ and output dimension $m$, where the number of neurons per hidden layer is $2n+2m+1$ or even $n+m+1$, see \cref {main_theorem_non_poly_classical_reg}.
The proofs in that case are based on the fact that for non-polyharmonic activation functions shallow CVNNs are universal, as shown by Voigtlaender in \cite[Theorem~1.3]{Voigtlaender2022}.
We combine this with an adaptation of the register model technique developed by Kidger and Lyons \cite[Theorem~3.2]{KidgerLy2020}.
There, they used this technique
to deduce universality of deep narrow \emph{real-valued} neural networks from the classical result on universality of shallow real-valued neural networks \cite{Cybenko1989,Pinkus1999}.
For polyharmonic activation functions, we show that CVNNs with input dimension $n$, output dimension $m$, and $2n+2m+5$ or even $n+m+3$ neurons per hidden layer are universal, see \cref{thm:main-polyharmonic}.
This is done by approximating polynomials in $z$ and $\overline{z}$ uniformly on compact sets and invoking the Stone--Weierstrass theorem.

Moreover, based on the ideas from \cite[Proposition~59]{kratsios2022universal}, we provide a quantitative approximation bound in terms of the depth of the considered networks for the case
of a smooth and non-polyharmonic activation function $\rr  \in C(\CC;\CC)$. 
Precisely, given a function $f \in C([-1,1]^n + \ii \cdot [-1,1]^n; \CC^m)$ and $\eps > 0$, we show that one can approximate $f$ up to precision $\eps$ with deep narrow networks using 
activation function $\rr$ and a depth of at most 
\begin{equation*}
\lr{ 32\cdot \left[\omega^{-1}\lr{f, \frac{\eps}{3 \cdot \sqrt{2m} \cdot \lr{1 + \frac{n}{2}}}}\right]^{-2} + 9}^{2n};
\end{equation*}
see \cref{thm:quant}.
Here, $\omega^{-1}(f, \cdot)$ denotes the inverse modulus of continuity of the function $f$, see \eqref{eq:inverse_mod}.
We believe that our techniques combined with those from \cite[Appendix~B.2]{kratsios2022universal} can also be used to derive depth estimates 
for deep narrow CVNNs for the case of more general activation functions 
$\rr \in C(\CC;\CC)$.
However, since the present paper focuses on the aspect of universality, this is left as future work.

\begin{figure}[H]
\centering
\begin{tikzpicture}[line cap=round,line join=round,>=stealth,x=0.8cm,y=0.8cm,post/.style={->, shorten >=0.5pt,shorten <=0.5pt}]
\begin{small}
\node (start) at (-2,0) {Is $\rr\in C(\CC;\CC)$ neither holomorphic, nor antiholomorphic, nor $\RR$-affine?};
\node (diff) at (-2,-2) {Is there $z_0\in\CC$ with $(\wirt\rr(z_0),\wirtq\rr(z_0))\neq (0,0)$?};
\node[draw,rectangle, text width = 2.1cm, align = center] (nonuniversal) at (5,-2) {no universality (\cref{thm:holo_anti_aff})};
\node (bothzero) at (-2,-4) {Is $\wirt\rr(z_0)=0$ or $\wirtq\rr(z_0)=0$?};
\node[draw,rectangle, text width = 2.1cm, align = center] (inconclusive) at (5,-4) {inconclusive (\cref{thm:necessary_diff})};
\node (poly1) at (-2,-6) {Is $\rr$ polyharmonic?};
\node (poly2) at (5,-6) {Is $\rr$ polyharmonic?};

\node[draw,rectangle, text width = 2.1cm, align = center] (width11) at (-3.5,-8) {$n+m+3$ (\cref{thm:main-polyharmonic})};
\node[draw,rectangle, text width = 2.1cm, align = center] (width12) at (-0.5,-8) {$n+m+1$ (\cref{main_theorem_non_poly_classical_reg})};
\node[draw,rectangle, text width = 2.1cm, align = center] (width21) at (3.5,-8) {$2n+2m+5$ (\cref{thm:main-polyharmonic})};
\node[draw,rectangle, text width = 2.1cm, align = center] (width22) at (6.5,-8) {$2n+2m+1$ (\cref{main_theorem_non_poly_classical_reg})};

\draw[post] (start)--(diff) node[midway,left]{yes};
\draw[post] (start)--(nonuniversal) node[midway,right=7pt]{no};
\draw[post] (diff)--(bothzero) node[midway,left]{yes};
\draw[post] (diff)--(inconclusive) node[midway,right=7pt]{no};
\draw[post] (bothzero)--(poly1) node[midway,left]{yes};
\draw[post] (bothzero)--(poly2) node[midway,right=7pt]{no};
\draw[post] (poly1)--(width11) node[midway,left]{yes};
\draw[post] (poly1)--(width12) node[midway,right]{no};
\draw[post] (poly2)--(width21) node[midway,left]{yes};
\draw[post] (poly2)--(width22) node[midway,right]{no};
\end{small}
\end{tikzpicture}
\caption{Our results in a nutshell.\label{fig:nutshell}}
\end{figure}
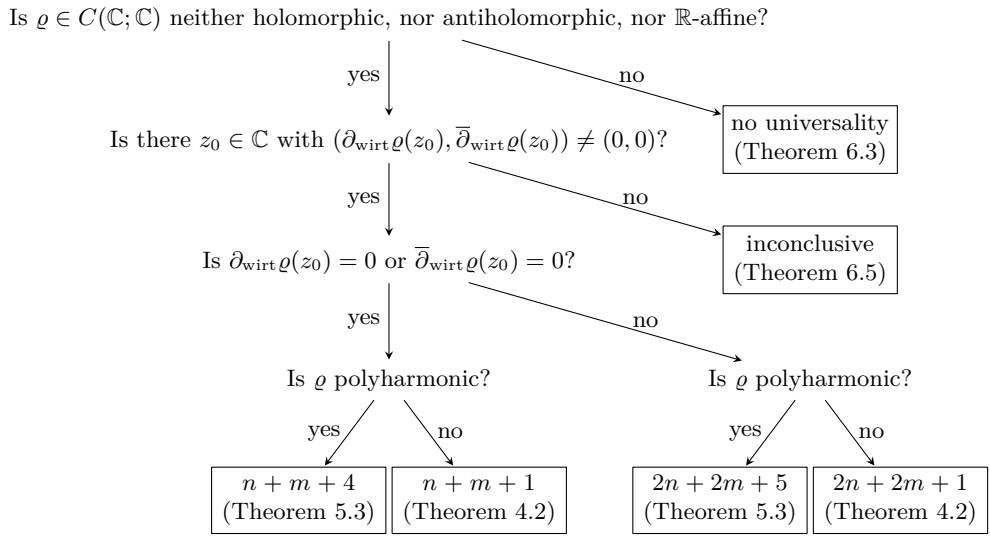

\subsection{Organization of our paper}
In \cref{chap:preliminaries} we fix our notation and recall some basics from complex and functional analysis and the theory of neural networks.
\cref{sec: building_blocks} introduces the register model and shows how the identity map on $\CC$ and complex conjugation can be approximated using CVNNs.
For non-polyharmonic functions, the proof of \cref{thm:intro} can be found in \cref{chap:non-polyharmonic}.
In \cref{chap:polyharmonic}, we present the proof of universality claimed in \cref{thm:intro} in the case of polyharmonic activation functions which are neither holomorphic, nor antiholomorphic, nor $\RR$-affine.
In \cref{sec:quant}, we prove the quantitative approximation bound in terms of the depth of the considered networks. 
In \cref{chap:optimality}, we show that CVNNs whose activation function is holomorphic or antiholomorphic or $\RR$-affine are never universal, 
regardless of the number of neurons per hidden layer.
Moreover, we show that there exist activation functions satisfying the assumptions from \cref{thm:intro} for which a width of $\max\setn{2n,2m}$ is necessary in order to provide 
universal CVNNs. 
In the appendix, we provide basics on the relationship between local uniform convergence and universal approximation, and on Taylor approximations in terms of Wirtinger derivatives.

\section{Preliminaries}\label{chap:preliminaries}
In this section, we recall facts from complex analysis, functional analysis, and the theory of neural networks behind the phrases in \cref{thm:intro}.
The presentation is loosely based on \cite[Chapter~7]{Rudin1976}, \cite[Chapter~11]{Rudin1987}, and \cite[Section~1]{GribonvalKuNiVo2022}.

\subsection{Complex and functional analysis}
We use the symbols $\NN$, $\RR$, and $\CC$ to denote the natural, real, and complex numbers, respectively.
By $\RE(z)$, $\IM(z)$, and $\overline{z}$, we denote the componentwise real part, imaginary part, and complex conjugate of a vector $z\in\CC^n$, respectively.
We call the function $\rr:\CC\to\CC$ \emph{partially differentiable} at $z_0$, if the \emph{partial derivatives}
\begin{align*}
\frac{\p \rr}{\p x}(z_0)\defeq\lim_{\RR\setminus \setn{0}\ni h\to 0}\frac{\rr(z_0+h)-\rr(z_0)}{h}\\
\text{and}\qquad\frac{\p \rr}{\p y}(z_0)\defeq\lim_{\RR\setminus\setn{0}\ni h\to 0}\frac{\rr(z_0+\ii h)-\rr(z_0)}{h}
\end{align*}
exist.
Higher-order partial derivates are defined in the standard manner.
We write $\rr\in C^k(\CC;\CC)$ if $\rr$ admits partial derivatives up to order $k$ at each point of $\CC$ and the $k$th-order partial derivatives are continuous functions $\CC\to\CC$.
Likewise, we write $\rr\in C^\infty(\CC;\CC)$ if $\rr\in C^k(\CC;\CC)$ for all $k\in\NN$.
If $\frac{\p \rr}{\p x}(z_0)$ and $\frac{\p \rr}{\p y}(z_0)$ exist and the identity
\begin{equation} \label{eq:real_diff}
\lim_{\CC\setminus\setn{0}\ni h\to 0}\frac{\rr(z_0+h)-\rr(z_0)-\frac{\p \rr}{\p x}(z_0)\RE(h)-\frac{\p \rr}{\p y}(z_0)\IM(h)}{h}=0
\end{equation}
holds true, then $\rr$ is called \emph{real differentiable at $z_0$} with derivative $(\frac{\p\rr}{\p x}(z_0),\frac{\p\rr}{\p y}(z_0))$.
Similarly, $\rr$ is called \emph{complex differentiable at $z_0$} if
\begin{equation*}
\lim_{\CC\setminus\setn{0}\ni h\to 0}\frac{\rr(z_0+h)-\rr(z_0)-ch}{h}=0
\end{equation*}
for some number $c\in\CC$, which in that case is given by
\begin{equation*}
\wirt \rr(z_0)\defeq \frac{1}{2}\lr{\frac{\p \rr}{\p x}(z_0)-\ii\frac{\p\rr}{\p y}(z_0)}.
\end{equation*}
Complex differentiability of $\rr$ at $z_0$ can be equivalently stated as
\begin{equation*}
\wirtq \rr(z_0)\defeq \frac{1}{2}\lr{\frac{\p \rr}{\p x}(z_0)+\ii\frac{\p\rr}{\p y}(z_0)}=0.
\end{equation*}
The differential operators $\wirt$ and $\wirtq$ are called \emph{Wirtinger derivatives}.
If $\wirtq\rr(z)=0$ for all $z\in\CC$, then $\rr$ is a \emph{holomorphic} function.
The function $\rr$ is called \emph{antiholomorphic} if the function $\overline{\rr}:\CC\to\CC$, $\overline{\rr}(z)\defeq\RE(\rr(z))- \ii \IM(\rr(z))$ is holomorphic or, equivalently, $\wirt\rr(z)=0$ for all $z\in\CC$.
As the linear operator that maps the partial derivatives onto the Wirtinger derivatives is invertible, it follows that $(\frac{\p\rr}{\p x}(z_0),\frac{\p\rr}{\p y}(z_0))=(0,0)$ if and only if $(\wirt\rr(z_0),\wirtq\rr(z_0))=(0,0)$.
Furthermore, the symmetry of mixed partial derivatives implies for $\rr\in C^2(\CC;\CC)$ that
\begin{equation}
4\wirt\wirtq\rr=4\wirtq\wirt\rr=\lr{\frac{\p^2}{\p x^2}+\frac{\p^2}{\p y^2}}\rr\eqdef \Delta \rr.\label{eq:laplace-wirtinger}
\end{equation}
If $\rr\in C^\infty(\CC;\CC)$ and $\Delta^m \rr=0$ for some $m\in\NN$, then $\rr$ is called \emph{polyharmonic of order $m$}.
Because of \eqref{eq:laplace-wirtinger}, holomorphic and antiholomorphic functions are \emph{harmonic}, i.e., polyharmonic of order $1$.

The following well-known lemma generalizes the classical real-valued Taylor expansion to the complex-valued setting; see \cref{app:taylor} for a proof.
\begin{Lem}
\label{thm:taylor}
Let $\rr\in C(\CC;\CC)$ and $z,z_0\in\CC$.
If $\rr$ is real differentiable at $z_0$, then
\begin{equation}
\label{eq:taylor-1}
\rr(z+z_0)=\rr(z_0)+\wirt\rr(z_0)z+\wirtq\rr(z_0)\overline{z}+\Theta_1(z)
\end{equation}
for a function $\Theta_1:\CC\to\CC$ with $\lim_{\CC\setminus\setn{0}\ni z\to 0}\frac{\Theta_1(z)}{z}=0$.
If  $\rr\in C^2(\CC;\CC)$, then
\begin{equation}
\label{eq:taylor-2}
\rr(z+z_0)=\rr(z_0)+\wirt\rr(z_0)z+\wirtq\rr(z_0)\overline{z}+\frac{1}{2}\wirt^2\rr(z_0)z^2+\wirt\wirtq\rr(z_0)z\overline{z}+\frac{1}{2}\wirtq^2\rr(z_0)\overline{z}^2+\Theta_2(z)
\end{equation}
for a function $\Theta_2:\CC\to\CC$ with $\lim_{\CC\setminus\setn{0}\ni z\to 0}\frac{\Theta_2(z)}{z^2}=0$.
\end{Lem}

On $\CC^m$, we consider the topology induced by the Euclidean norm
\begin{equation*}\mnorm{(z_1,\ldots,z_m)}_{\CC^m}=\lr{\sum_{j=1}^m \abs{z_j}^2}^{\frac{1}{2}}.
\end{equation*}
For $K\subseteq \CC^n$, we denote the vector space of continuous functions $K\to\CC^m$ by $C(K;\CC^m)$.
When $K\subseteq \CC^n$ is compact, the expression
\begin{equation*}
\mnorm{f}_{C(K;\CC^m)}\defeq\sup_{z\in K}\mnorm{f(z)}_{\CC^m}
\end{equation*}
defines a norm on $C(K;\CC^m)$, called the \emph{uniform norm}, which renders $C(K;\CC^m)$ a Banach space.
The convergence of a sequence $(f_j)_{j\in\NN}$ of elements $f_j\in C(K;\CC^m)$ to a limit $f\in C(K;\CC^m)$
with respect to $\mnorm{\cdot}_{C(K; \CC^m)}$ is written as \emph{$f_j\to f$ as $j\to\infty$, uniformly on $K$}.
Similarly, a sequence $(f_j)_{j \in \NN}$ of functions $f_j \in C(\CC^n; \CC^m)$ is said to converge \emph{locally uniformly} to a function $f \in C(\CC^n; \CC^m)$ if it converges uniformly to $f$ on every compact subset $K \subseteq \CC^n$.
Since compositions of continuous functions are continuous, we have $\mathcal{NN}^\rr_{n,m,W}\subseteq C(\CC^n;\CC^m)$ when $\rr\in C(\CC;\CC)$.
The main objective of the paper at hand is to show that under certain additional assumptions on $\rr$ and $W$, the elements of $\mathcal{NN}^\rr_{n,m,W}$ are arbitrarily close to the elements of $C(\CC^n;\CC^m)$ in the following sense.
\begin{Def}
\label{def:universal_approximation_property}
We say that a function class $\mathcal{F} \subseteq C(\CC^n;\CC^m)$ has the \emph{universal approximation property} or is \emph{universal}) if for every function $g \in C(\CC^n;\CC^m)$, every compact subset $K \subseteq \CC^n$ and every $\eps>0$ there exists a function $f\in\mathcal{F}$ such that
\begin{equation*}
\sup_{z\in K}\mnorm{f(z)-g(z)}_{\CC^m} < \eps.
\end{equation*}
\end{Def}
The universal approximation property of $\mathcal{F}$ is equivalent to saying that for every
given function $g \in C(\CC^n;\CC^m)$ there exists a sequence $(f_j)_{j \in \NN}$ with $f_j \in \mathcal{F}$ for every $j \in \NN$ that converges locally uniformly to $g$.
Likewise, this is equivalent to saying that the class $\mathcal{F}$ is dense in $C(\CC^n; \CC^m)$ with respect to the \emph{compact-open topology}.
We elaborate this equivalence in \cref{loc_uniform_convergence}. 
By $\overline{\mathcal{F}}$ we denote the closure of $\mathcal{F}$ with respect to the compact-open topology. 
Notice that we also use the notation $\overline{z}$ to denote the complex conjugate of $z \in \CC$.
The precise meaning will be clear from the context.

\subsection{Neural networks}
A (fully connected feed-forward) \emph{complex-valued neural network} (CVNN) is a function
\begin{equation*}
 V_L\circ \rr^{\times N_{L}} \ldots \circ \rr^{\times N_1}\circ V_0:\quad \CC^{N_0}\to\CC^{N_{L+1}}
\end{equation*}
where
\begin{itemize}
\item{$L\in\NN$ is called the \emph{depth} of the CVNN,}
\item{$N_j\in \NN$ is the \emph{width} of the $j$th layer,}
\item{$\max\setn{N_0,\ldots,N_{L+1}}$ is the \emph{width} of the CVNN,}
\item{$V_j:\CC^{N_{j}}\to \CC^{N_{j+1}}$ is a $\CC$-affine map, 
abbreviated as $V_j\in\Aff(\CC^{N_{j}};\CC^{N_{j+1}})$, i.e., there exist $A_j\in\CC^{N_{j+1}\times N_{j}}$ 
and $b_j\in\CC^{N_{j+1}}$ such that $V_j(z)=A_jz+b_j$ for all $z\in\CC^{\NN_{j}}$,}
\item{$\rr^{\times N_j}(z_1,\ldots,z_{N_j})=(\rr(z_1),\ldots,\rr(z_{N_j}))$ is the componentwise application of a (potentially non-affine) map $\rr:\CC\to\CC$ called the \emph{activation function}.}
\end{itemize}
We refer to the numbers $N_0$ and $N_{L+1}$ as the \emph{input dimension} and \emph{output dimension}, respectively.
The layers $1,\ldots,L$ are the \emph{hidden layers} of the CVNN.
Since it is always possible to pad matrices and vectors by additional zero rows and columnns, we may and will assume without loss of generality that $N_1=N_2=\ldots=N_{L}$.

We introduce a short-hand notation for the CVNNs that arise this way.
\begin{Def}\label{def:cvnn}
Let $n,m,W,L \in \NN$ and $\rr: \CC \to \CC$.
We denote by $\mathcal{NN}^\rr_{n,m,W,L}$ the set of CVNNs with depth $L$, input dimension $n$, output dimension $m$, and $N_j=W$ for $j\in\setn{1,\ldots,L}$.
In view of cases where the depth is not relevant, we let 
\begin{equation*}
\mathcal{NN}^{\rr}_{n,m,W} \defeq \bigcup_{L \in \NN} \mathcal{NN}^\rr_{n,m,W,L}.
\end{equation*}
\end{Def}
The elements of $\mathcal{NN}^\rr_{n,m,W}$ are thus alternating compositions
\begin{equation}
V_L \circ \rr^{\times W} \circ \ldots \circ \rr^{\times W} \circ V_0:\quad \CC^n\to\CC^m\label{eq:composition}
\end{equation}
of $\CC$-affine maps $V_j$ and the componentwise applications of the activation function $\rr$.
In the subsequent \cref{chap:polyharmonic,chap:non-polyharmonic}, we have $W\geq \max\setn{n,m}$, such that $W$ turns out to be the width of the neural networks under consideration.

A typical way of thinking about neural networks is viewing the component functions of $\rr^{\times N_{j+1}}\circ V_j$  as building blocks called \emph{neurons}.
Each neuron performs a computation of the form
\begin{equation*}
z\mapsto \rr(w^\top z + b)
\end{equation*}
where $z$ is the output of the previous layer, $w$ a vector of \emph{weights}, and $b$ a number called \emph{bias}.

Since the composition of affine maps is affine, it is also possible to think about neural networks as maps
\begin{align*}
	\big( \Psi_L \circ \rr^{\times W} \circ \Phi_{L} \big) \circ \big( \Psi_{L-1} \circ \rr^{\times W} \circ \Phi_{L-1} \big)\circ\ldots \circ \big( \Psi_1 \circ \rr^{\times W} \circ \Phi_1 \big),
\end{align*}
where each of the maps $\Phi_k, \Psi_k$ is affine. 
This allows to perceive shallow networks, see \cref{def:shallow_cvnn}, as building blocks for neural networks.
This point of view is similar to the notion of \emph{enhanced neurons} in \cite{KidgerLy2020}.

For a fixed activation function $\rr$, different choices of the $\CC$-affine functions $V_j$ may lead to the same composite function \eqref{eq:composition}.
In view of this, both depth and width of a CVNN are not properties of the function \eqref{eq:composition} but of the tuple $(V_1,\ldots,V_L)$.
For this reason, a different terminology is sometimes used in the literature where $(V_1,\ldots,V_L)$ is called the \emph{neural network} and \eqref{eq:composition} is its \emph{realization}, cf.\ \cite[Section~1]{GribonvalKuNiVo2022}.

Apart from the choice of the activation function $\rr$, restrictions on the depth or the width are common ingredients in the analysis of (fully connected feed-forward) neural networks.
A CVNN is called \emph{shallow} if its depth equals $1$, and \emph{deep} otherwise.
Since shallow networks play a special role in our analysis, we introduce an own notation for them. 
\begin{Def}\label{def:shallow_cvnn}
Let $n,m,W \in \NN$ and $\rr: \CC \to \CC$.
We denote by 
\begin{equation*}
\mathcal{SN}^\rr_{n,m,W} \defeq \mathcal{NN}^\rr_{n,m,W,1}
\end{equation*}
the set of \emph{shallow CVNNs} with $W$ hidden neurons.
We write 
\begin{equation*}
\mathcal{SN}^\rr_{n,m} \defeq \bigcup_{W \in \NN} \mathcal{SN}^\rr_{n,m,W}.
\end{equation*}
\end{Def}
In contrast to shallowness, \emph{narrowness} is not an individual property of CVNNs but a class property.
A set $\mathcal{F}$ of CVNNs is said to be \emph{narrow} if it does not contain CVNNs of arbitrarily large widths, i.e., 
if $\mathcal{F}\subseteq\mathcal{NN}^\rr_{n,m,W}$ for suitable $n,m,W\in\NN$ and $\rr:\CC\to\CC$.

\section{Building blocks and register model}
\label{sec: building_blocks}
In this section, we prove several preliminary statements that are crucial for the results derived in \cref{sec:quant,sec:main}.
Specifically, in \cref{subsec: building_blocks}, we construct building blocks to approximate elementary functions locally uniformly by shallow CVNNs. 
In \cref{subsec:register}, we consider the fundamental concept of the register model to transform shallow networks into deep narrow networks. 
\subsection{Building blocks}\label{subsec: building_blocks}

In this section we introduce various \emph{building blocks} for complex-valued networks, i.e., small neural network blocks that are able to represent elementary functions (e.g., the complex identity $\id_\CC$ or complex conjugation $\overline{\id_\CC}$) up to an arbitrarily small approximation error.
These building blocks are used in \cref{chap:polyharmonic,chap:non-polyharmonic} to construct the deep narrow networks that we use to approximate a given continuous function.
Throughout the chapter, we assume that the used activation function $\rr: \CC \to \CC$ is differentiable (in the real sense) at one point with non-vanishing derivative at that point.
In fact, the strategy for constructing these building blocks is always going to be similar: By using the first- and second-order Taylor expansion of the activation function $\rr$ as introduced in \cref{thm:taylor} one can localize the activation function around its point of differentiability where it behaves like a complex polynomial in $z$ and $\overline{z}$ of degree $1$ and $2$, respectively.
This enables us to extract elementary functions from that Taylor expansion.

\cref{prop: main} is fundamental for the universality results introduced in \cref{chap:polyharmonic,chap:non-polyharmonic}.
It shows that it is possible to uniformly approximate the complex identity or complex conjugation on compact sets using neural networks with a single hidden layer and width at most $2$.
If the activation function (or its complex conjugate) is not just real but even complex differentiable, the width can be reduced to $1$.
See \cref{fig:main} for an illustration of the building blocks.
\begin{Prop} \label{prop: main}
Let $\rr \in C(\CC;\CC)$.
Assume that there exists $z_0 \in \CC$ such that $\rr$ is real differentiable at $z_0$ with $(\wirt \rr(z_0),\wirtq\rr(z_0))\neq (0,0)$.
Furthermore, let $K \subseteq \CC$ be compact and $\eps > 0$.
\begin{enumerate}[label={(\roman*)},leftmargin=*,align=left,noitemsep]
\item{\label{prop: main item 1} If $\wirt \rr (z_0) \neq 0$ and $\wirtq \rr (z_0) = 0$, there exist $\phi, \psi \in \Aff(\CC;\CC)$ such that
\begin{equation*}
\sup_{z\in K}\abs{(\psi \circ \rr \circ \phi)(z) - z } < \eps.
\end{equation*}}
\item{\label{prop: main item 2} If $\wirt \rr (z_0) = 0$ and $\wirtq \rr (z_0) \neq 0$, there exist $\phi, \psi \in \Aff(\CC;\CC)$ such that
\begin{equation*}
\sup_{z\in K}\abs{ (\psi \circ \rr \circ \phi)(z) - \overline{z}} < \eps.
\end{equation*}}
\item{\label{prop: main item 3} If $\wirt \rr (z_0) \neq 0 \neq \wirtq \rr(z_0)$, there exist $\phi \in \Aff(\CC;\CC^2)$ and $\psi \in \Aff(\CC^2;\CC^2)$ such that
\begin{equation*}
\sup_{z\in K}\mnorm{ (\psi \circ \rr^{\times 2} \circ \phi)(z) - (z,\overline{z})}_{\CC^2} < \eps
\end{equation*}}
\end{enumerate}
\end{Prop}
\begin{proof}
	Recall that \cref{thm:taylor} yields the existence of a function $\Theta: \CC \to \CC$ satisfying $\lim\limits_{z \to 0} \frac{\Theta(z)}{z} = 0$ and
	\begin{equation*}
	\rr(z + z_0) = \rr(z_0)  + \wirt \rr (z_0) z + \wirtq \rr (z_0) \overline{z}+ \Theta(z)
	\end{equation*}
	for every $z \in \CC$.
	
	If $\wirtq \rr (z_0) = 0$ and $\wirt \rr(z_0) \neq 0$ we see for all $h > 0$ and $z \in K$ that
	\begin{equation} \label{eq:approx_ident}
		\frac{\rr(z_0 + hz) - \rr(z_0)}{\wirt \rr(z_0) h} = z + \frac{\Theta(hz)}{\wirt \rr (z_0) h }.
	\end{equation}
	Similarly, if $\wirtq \rr (z_0) \neq 0$ and $\wirt \rr(z_0) = 0$, we get for all $h>0$ and $z\in K$ that
	\begin{equation} \label{eq:approx_conj}
	 \frac{\rr(z_0 + hz) - \rr(z_0)}{\wirtq \rr(z_0) h} = \overline{z} + \frac{\Theta(hz)}{\wirtq \rr (z_0) h }.
	\end{equation}
	If $\wirt \rr (z_0) \neq 0 \neq \wirtq \rr (z_0)$, consider
	\begin{equation} \label{eq:approx_ident_2}
	\frac{\ii \rr(z_0 + hz) + \rr(z_0 + \ii hz)- (1+ \ii) \rr (z_0) }{2 \ii h \wirt \rr (z_0)} = z + \frac{\ii\Theta(hz) + \Theta(\ii hz)}{2 \ii h \wirt \rr (z_0)}
	\end{equation}
	as well as
	\begin{equation} \label{eq:approx_conj_2}
	 \frac{-\ii \rr(z_0 + hz) + \rr(z_0 + \ii hz)-(1-\ii)\rr (z_0)}{-2 \ii h \wirtq \rr (z_0)} = \overline{z} + \frac{-\ii\Theta(hz) + \Theta(\ii hz)}{-2 \ii h \wirtq \rr (z_0)}.
	\end{equation}
	It remains to show that the second summands on the right-hand sides of \eqref{eq:approx_ident}, \eqref{eq:approx_conj}, \eqref{eq:approx_ident_2}, and \eqref{eq:approx_conj_2} tend to $0$ as $h\downarrow 0$.
	Since $K$ is compact there exists $L > 0$ satisfying $\abs{z} \leq L$ for all $z \in K$.
	Let $\eps' > 0$ be arbitrary and take $\delta > 0$ such that
	\begin{equation*}
	\abs{\frac{\Theta(w)}{w}} < \frac{\eps'}{L}
	\end{equation*}
	for every $w \in \CC \setminus \setn{0}$ with $\abs{w} < \delta$.
	Let $h \in (0, \delta /L)$ and $z \in K$.
	Since $\abs{hz} < \delta$ we see for every $z \in K \setminus \setn{0}$ that
	\begin{equation*}
	\abs{\frac{\Theta(hz)}{h}} \leq L \cdot \abs{\frac{\Theta (hz)}{hz}} < \eps'
	\end{equation*}
	and since $\eps'$ has been taken arbitrarily
	\begin{equation*}
	\lim_{h \downarrow 0}\sup_{z \in K} \frac{\Theta(hz)}{h} = 0.
	\end{equation*}
	Here, we concluded $\Theta (0) =0$ from \eqref{eq:taylor-1} to also cover the case $z=0$.
\end{proof}

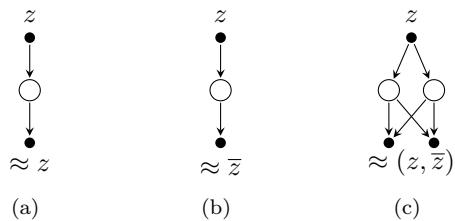
\begin{figure}[H]
 \centering
\subfigure[\label{z}]{
\begin{tikzpicture}[x=1.0cm,y=0.7cm,>=stealth,neuron/.style={circle,draw=black,inner sep=0pt,minimum size=8pt},
transition/.style={circle,fill=black,inner sep=0pt,minimum size=4pt},
pre/.style={->, shorten >=0.5pt,shorten <=0.5pt},
post/.style={->, shorten >=0.5pt,shorten <=0.5pt},
peri/.style={->, shorten >=0.5pt,shorten <=0.5pt}]
\clip (-0.8,-2.75) rectangle (0.8,0.55);

\node[transition,label=above:$z$] (in1) at (0,0){};
\node[neuron] (hidden1) at (0,-1){};
\node[transition,label=below:$\approx z$] (out1) at (0,-2){};

\draw[pre] (in1)--(hidden1);
\draw[post] (hidden1)--(out1);
\end{tikzpicture}
}\qquad
\subfigure[\label{zquer}]{
\begin{tikzpicture}[x=1.0cm,y=0.7cm,>=stealth,neuron/.style={circle,draw=black,inner sep=0pt,minimum size=8pt},
transition/.style={circle,fill=black,inner sep=0pt,minimum size=4pt},
pre/.style={->, shorten >=0.5pt,shorten <=0.5pt},
post/.style={->, shorten >=0.5pt,shorten <=0.5pt},
peri/.style={->, shorten >=0.5pt,shorten <=0.5pt}]
\clip (-0.8,-2.75) rectangle (0.8,0.55);

\node[transition,label=above:$z$] (in1) at (0,0){};
\node[neuron] (hidden1) at (0,-1){};
\node[transition,label=below:$\approx \overline{z}$] (out1) at (0,-2){};

\draw[pre] (in1)--(hidden1);
\draw[post] (hidden1)--(out1);
\end{tikzpicture}
}\qquad
\subfigure[\label{zzquer}]{
\begin{tikzpicture}[x=1.0cm,y=0.7cm,>=stealth,neuron/.style={circle,draw=black,inner sep=0pt,minimum size=8pt},
transition/.style={circle,fill=black,inner sep=0pt,minimum size=4pt},
pre/.style={->, shorten >=0.5pt,shorten <=0.5pt},
post/.style={->, shorten >=0.5pt,shorten <=0.5pt},
peri/.style={->, shorten >=0.5pt,shorten <=0.5pt}]
\clip (-0.8,-2.75) rectangle (0.8,0.55);

\node[transition,label=above:$z$] (in1) at (0,0){};
\node[neuron] (hidden1) at (-0.3,-1){};
\node[neuron] (hidden2) at (0.3,-1){};
\node[transition] (out1) at (-0.3,-2){};
\node[transition] (out2) at (0.3,-2){};

\draw[pre] (in1)--(hidden1);
\draw[pre] (in1)--(hidden2);
\draw[post] (hidden1)--(out1);
\draw[post] (hidden2)--(out1);
\draw[post] (hidden1)--(out2);
\draw[post] (hidden2)--(out2);
\draw ($(0,-2)+(0,-0.8)$) node[above]{$\approx (z,\overline{z})$};
\end{tikzpicture}
}
\caption{Illustration of the neural network building blocks from \cref{prop: main}.
Neurons in the input and output layers are depicted in filled dots at the top and bottom, respectively.
Applications of the activation function $\rr$ are shown as circles.\label{fig:main}}
\end{figure}

\cref{prop: approx} is important for the case of polyharmonic activation functions which is considered in \cref{chap:polyharmonic}.
It essentially states that, given an activation function which is not $\RR$-affine, 
one can approximate one of the functions $z \mapsto z \overline{z}$, $\ z \mapsto z^2$ or $z \mapsto \overline{z}^2$ by
 using a shallow neural network of width $4$, see \cref{fig:approx} for an illustration.
\begin{Prop} \label{prop: approx}
Let 
\begin{equation*}
f_1: \quad \CC \to \CC, \quad f_1(z) = z \overline{z}, \quad f_2: \quad \CC \to \CC, \quad f_2(z) = z^2, \quad \text{and} \quad 
f_3: \quad \CC \to \CC, \quad f_3(z) = \overline{z}^2.
\end{equation*}
Moreover, let $\rr \in C^2(\CC;\CC)$ be not $\RR$-affine.
Then there exists a function $f  \in \setn{f_1,f_2,f_3}$ with the following property:
For every compact subset $K \subseteq \CC$ and every $\eps>0$
there exist affine maps $\phi \in \Aff(\CC;\CC^4)$ and $\psi \in \Aff(\CC^4;\CC)$ such that 
\begin{align*}
\sup_{z\in K}\abs{(\psi \circ \rr^{\times 4} \circ \phi) (z) - f(z)}&< \eps
\end{align*}
holds true.
\end{Prop}
\begin{proof}
Since $\rr$ is not $\RR$-affine, there exists $z_0 \in \CC$ such that $\wirt^2 \rr (z_0) \neq 0$, $\wirt \wirtq \rr (z_0) \neq 0$, or $\wirtq^2\rr(z_0) \neq 0$, see, e.g., \cref{app:helpproof}.
Using the second-order Taylor expansion stated in \cref{thm:taylor}, there exists a function $\Theta: \CC \to \CC$ satisfying $\lim\limits_{w \to 0} \frac{\Theta(w)}{w^2} = 0$ and
\begin{align*}
\rr(z_0+w) = \rr (z_0) &+ \wirt \rr (z_0) w + \wirtq \rr(z_0) \overline{w} + \frac{1}{2} \wirt^2 \rr (z_0)w^2 + \wirt \wirtq \rr (z_0)w \overline{w}\\
& + \frac{1}{2} \wirtq^2 \rr(z_0) \overline{w}^2 + \Theta(w)
\end{align*}
for every $w \in \CC$.
Applying this identity to $-w$ in place of $w$ and adding up, we infer for any $w \in \CC$ that
\begin{equation*}
\rr (z_0 + w) + \rr (z_0 - w) = 2\rr(z_0) + \wirt^2 \rr(z_0) w^2 + 2 \wirt \wirtq \rr(z_0) w \overline{w} + \wirtq^2\rr (z_0) \overline{w}^2 + \Theta(w) + \Theta(-w).
\end{equation*}

Let $h > 0$ and $z \in K$.
If $\wirt \wirtq \rr (z_0) \neq 0$, we see with $w=hz$ and $w=\ii hz$ that

\begin{align}
&\norel \frac{\rr(z_0 + hz) + \rr (z_0 - hz) + \rr(z_0 + \ii hz) + \rr(z_0 - \ii hz) -4\rr(z_0) }{4 h^2 \wirt\wirtq \rr(z_0)} \nonumber\\
\label{eq:z^2}
&=z \overline{z} + \frac{\Theta(hz) + \Theta(-hz) + \Theta(\ii hz) + \Theta(- \ii hz)}{4 h^2 \wirt\wirtq \rr (z_0)}.
\end{align}
If $\wirt \wirtq \rr (z_0) = 0$ and $\wirt^2 \rr (z_0) \neq 0$, consider $w=hz$ and $w=\sqrt{\ii}hz$, where $\sqrt{\ii}$ is a fixed square root of $\ii$:
\begin{align}
&\norel \frac{\rr(z_0 +hz) + \rr(z_0 - hz) -\ii \rr(z_0 + \sqrt{\ii}hz) - \ii\rr(z_0 - \sqrt{\ii}hz) + 2(-1+\ii) \rr(z_0) }{2h^2 \wirt^2 \rr(z_0)}\nonumber \\
\label{eq:zzbar}
&=z^2 + \frac{\Theta(hz) + \Theta(-hz) -\ii \Theta (\sqrt{\ii}hz) -\ii \Theta(-\sqrt{\ii}hz)}{2h^2 \wirt^2 \rr(z_0)}.
\end{align}
Last, if $\wirt^2 \rr(z_0) = \wirt \wirtq \rr(z_0) = 0$, consider $w=hz$:
\begin{equation} \label{eq:zbar^2}
\frac{\rr(z_0 + hz) + \rr(z_0 - hz)- 2\rr(z_0) }{h^2 \wirtq^2 \rr (z_0)}= \overline{z}^2 + \frac{\Theta(hz) + \Theta(-hz)}{h^2 \wirtq^2 \rr(z_0)}.
\end{equation}
It remains to show that the second summands on the right-hand sides of \eqref{eq:z^2}, \eqref{eq:zzbar}, and \eqref{eq:zbar^2} tend to $0$ as $h\downarrow 0$.
Since $K$ is bounded, there exists $L> 0$ with $\abs{z} \leq L$ for every $z \in K$.
For given $\eps' > 0$ there exists $\delta > 0$ such that
\begin{equation*}
\abs{\frac{\Theta(w)}{w^2}} < \frac{\eps'}{L^2}
\end{equation*}
for every $w \in \CC \setminus \setn{0}$ with $\abs{w} < \delta$.
Hence, we see for every $h \in (0, \delta/L)$ and all $z \in K \setminus \setn{0}$ that
\begin{equation*}
\abs{\frac{\Theta(hz)}{h^2}} \leq L^2 \cdot \abs{\frac{\Theta(hz)}{(hz)^2}} < \eps'
\end{equation*}
where we used that $\abs{hz} < \delta$.
Therefore, we conclude
\begin{equation*}
	\lim_{h \downarrow 0}\sup_{z \in K} \frac{\Theta(hz)}{h^2} = 0,
\end{equation*}
using $\Theta(0)= 0$ from \eqref{eq:taylor-2} to also cover the case $z=0$.
\end{proof}

\begin{figure}[H]
\centering
\begin{tikzpicture}[x=1.0cm,y=0.8cm,>=stealth,neuron/.style={circle,draw=black,inner sep=0pt,minimum size=8pt},
transition/.style={circle,fill=black,inner sep=0pt,minimum size=4pt},
pre/.style={->, shorten >=0.5pt,shorten <=0.5pt},
post/.style={->, shorten >=0.5pt,shorten <=0.5pt},
peri/.style={->, shorten >=0.5pt,shorten <=0.5pt}]

\node[transition,label=above:$z$] (in1) at (0,0){};
\node[neuron] (hidden1) at (-0.9,-1){};
\node[neuron] (hidden2) at (-0.3,-1){};
\node[neuron] (hidden3) at (0.3,-1){};
\node[neuron] (hidden4) at (0.9,-1){};
\node[transition,label=below:{$\approx z^2$ or $\approx \overline{z}^2$ or $\approx z\overline{z}$}] (out1) at (0,-2){};

\draw[pre] (in1)--(hidden1);
\draw[pre] (in1)--(hidden2);
\draw[pre] (in1)--(hidden3);
\draw[pre] (in1)--(hidden4);
\draw[post] (hidden1)--(out1);
\draw[post] (hidden2)--(out1);
\draw[post] (hidden3)--(out1);
\draw[post] (hidden4)--(out1);
\end{tikzpicture}
\caption{Illustration of the neural network building block from \cref{prop: approx}.
Neurons in the input and output layers are depicted in filled dots at the top and bottom, respectively.
Applications of the activation function $\rr$ are shown as circles.\label{fig:approx}}
\end{figure}
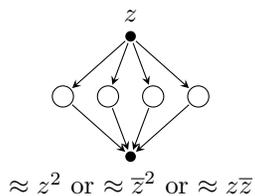

In order to approximate arbitrary polynomials in the variables $z_1,\ldots,z_n$ and $\overline{z_1},\ldots,\overline{z_n}$, we will compute iterative products of two complex numbers in \cref{thm:main-polyharmonic}.
The following result enables the approximation of such products.
An illustration of the CVNN blocks appearing in the proof are given in \cref{fig: multiplication_approx}.
\begin{Prop} \label{prop: multiplication_approx} \label{prop: mul_approx}
Let 
\begin{align*}
\mul_1:\CC\times \CC\to\CC,\qquad &\mul_1(z,w)=zw,\\
\mul_2:\CC\times \CC\to\CC,\qquad &\mul_2(z,w)=z\overline{w},\\
\mul_3:\CC\times \CC\to\CC,\qquad &\mul_3(z,w)=\overline{z}\overline{w}.
\end{align*}
Moreover, let $\rr \in C^2(\CC;\CC)$ be not $\RR$-affine.
Then there exists $\mul  \in \setn{\mul_1, \mul_2, \mul_3}$ with the following property: 
For every compact subset $K \subseteq \CC^2$ and $\eps > 0$
there exist $\phi \in \Aff(\CC^2;\CC^{12})$ and $\psi \in \Aff(\CC^{12};\CC)$ such that 
\begin{align*}
\sup_{(z,w)\in K}\abs{(\psi \circ \rr^{\times 12} \circ \phi) (z, w) - \mul(z,w) } &< \eps
\end{align*}
holds true.
\end{Prop}
\begin{proof}
 The main steps of the proof are to use a variant of the polarization identity to reconstruct the three multiplication operators from their values on the diagonal, and then to apply \cref{prop: approx} to approximate the latter.

Precisely, the construction is as follows: If $\zeta \mapsto \zeta \overline{\zeta}= \abs{\zeta}^2$ can be approximated according to the first case of \cref{prop: approx}, use the identity
 \begin{align*}
\lr{\frac{1}{4} + \frac{\ii}{4}} \abs{z + w}^2 + \lr{-\frac{1}{4} + \frac{\ii}{4}} \abs{z - w}^2 - \frac{\ii}{2} \abs{z - \ii w}^2 = z\overline{w}.
\end{align*}
Thus, in order to approximate $(z,w) \mapsto z \overline{w}$, one needs $4$ hidden neurons to approximate $\zeta \mapsto \abs{\zeta}^2$ for each of the $3$ linear combinations of $z$ and $w$, resulting in a total amount of $12$ hidden neurons.

If we have the second case of \cref{prop: approx}, we can approximate $\zeta \mapsto \zeta^2$ using $4$ hidden neurons.
In this case, consider
\begin{align*}
\frac{1}{4} \left[ (z+w)^2 - (z-w)^2\right] = zw,
\end{align*}
so that in total one needs $8$ neurons in order to approximate $(z,w) \mapsto zw$.
In the last case of \cref{prop: approx} we can approximate $\zeta \mapsto \overline{\zeta}^2$ using $4$ hidden neurons.
Considering
\begin{equation*}
\frac{1}{4} \left[\overline{(z+w)}^2 -  \overline{(z-w)}^2\right] = \overline{zw},
\end{equation*}
we infer that $(z,w) \mapsto \overline{zw}$ can be approximated using $8$ hidden neurons.
\end{proof}
It remains open whether the number $12$ in \cref{prop: multiplication_approx} can be reduced.
However, it should be noted that the exact number is not crucial for the final result \cref{thm:main-polyharmonic}, since in its proof
the shallow networks obtained in \Cref{prop: multiplication_approx}
are transformed to deep narrow networks according to \Cref{register_model,prop:register_to_full} and the width of these networks is in fact independent of 
the number of hidden neurons in the original
shallow networks.

\begin{figure}[H]
\centering
\subfigure[\label{threecomb}]{
\centering
\begin{tikzpicture}[x=0.9cm,y=0.63cm,>=stealth,neuron/.style={circle,draw=black,inner sep=0pt,minimum size=8pt},
transition/.style={circle,fill=black,inner sep=0pt,minimum size=4pt},
pre/.style={->, shorten >=0.5pt,shorten <=0.5pt},
post/.style={->, shorten >=0.5pt,shorten <=0.5pt},
peri/.style={->, shorten >=0.5pt,shorten <=0.5pt}]

\node[transition,label=above:$z$] (init1) at (1.2,1){};
\node[transition,label=above:$w$] (init2) at (3.6,1){};
\node[transition] (in1) at (0,0){};
\node[transition] (in2) at (2.4,0){};
\node[transition] (in3) at (4.8,0){};
\node[neuron] (hidden1) at (-0.9,-1){};
\node[neuron] (hidden2) at (-0.3,-1){};
\node[neuron] (hidden3) at (0.3,-1){};
\node[neuron] (hidden4) at (0.9,-1){};
\node[neuron] (hidden5) at (1.5,-1){};
\node[neuron] (hidden6) at (2.1,-1){};
\node[neuron] (hidden7) at (2.7,-1){};
\node[neuron] (hidden8) at (3.3,-1){};
\node[neuron] (hidden9) at (3.9,-1){};
\node[neuron] (hidden10) at (4.5,-1){};
\node[neuron] (hidden11) at (5.1,-1){};
\node[neuron] (hidden12) at (5.7,-1){};
\node[transition] (out1) at (0,-2){};
\node[transition] (out2) at (2.4,-2){};
\node[transition] (out3) at (4.8,-2){};
\node[transition,label=below:{$\approx z\overline{w}$}] (final) at (2.4,-3){};

\draw[peri] (init1)--(in1);
\draw[peri] (init1)--(in2);
\draw[peri] (init1)--(in3);
\draw[peri] (init2)--(in1);
\draw[peri] (init2)--(in2);
\draw[peri] (init2)--(in3);
\draw[pre] (in1)--(hidden1);
\draw[pre] (in1)--(hidden2);
\draw[pre] (in1)--(hidden3);
\draw[pre] (in1)--(hidden4);
\draw[pre] (in2)--(hidden5);
\draw[pre] (in2)--(hidden6);
\draw[pre] (in2)--(hidden7);
\draw[pre] (in2)--(hidden8);
\draw[pre] (in3)--(hidden9);
\draw[pre] (in3)--(hidden10);
\draw[pre] (in3)--(hidden11);
\draw[pre] (in3)--(hidden12);
\draw[post] (hidden1)--(out1);
\draw[post] (hidden2)--(out1);
\draw[post] (hidden3)--(out1);
\draw[post] (hidden4)--(out1);
\draw[post] (hidden5)--(out2);
\draw[post] (hidden6)--(out2);
\draw[post] (hidden7)--(out2);
\draw[post] (hidden8)--(out2);
\draw[post] (hidden9)--(out3);
\draw[post] (hidden10)--(out3);
\draw[post] (hidden11)--(out3);
\draw[post] (hidden12)--(out3);
\draw[peri] (out1)--(final);
\draw[peri] (out2)--(final);
\draw[peri] (out3)--(final);
\end{tikzpicture}
}\qquad
\subfigure[\label{twocomb}]{
\centering
\begin{tikzpicture}[x=0.9cm,y=0.63cm,>=stealth,neuron/.style={circle,draw=black,inner sep=0pt,minimum size=8pt},
transition/.style={circle,fill=black,inner sep=0pt,minimum size=4pt},
pre/.style={->, shorten >=0.5pt,shorten <=0.5pt},
post/.style={->, shorten >=0.5pt,shorten <=0.5pt},
peri/.style={->, shorten >=0.5pt,shorten <=0.5pt}]

\node[transition,label=above:$z$] (init1) at (0,1){};
\node[transition,label=above:$w$] (init2) at (2.4,1){};
\node[transition] (in1) at (0,0){};
\node[transition] (in2) at (2.4,0){};
\node[neuron] (hidden1) at (-0.9,-1){};
\node[neuron] (hidden2) at (-0.3,-1){};
\node[neuron] (hidden3) at (0.3,-1){};
\node[neuron] (hidden4) at (0.9,-1){};
\node[neuron] (hidden5) at (1.5,-1){};
\node[neuron] (hidden6) at (2.1,-1){};
\node[neuron] (hidden7) at (2.7,-1){};
\node[neuron] (hidden8) at (3.3,-1){};
\node[transition] (out1) at (0,-2){};
\node[transition] (out2) at (2.4,-2){};
\node[transition,label=below:{$\approx zw$ or $\approx \overline{zw}$}] (final) at (1.2,-3){};

\draw[peri] (init1)--(in1);
\draw[peri] (init1)--(in2);
\draw[peri] (init2)--(in1);
\draw[peri] (init2)--(in2);
\draw[pre] (in1)--(hidden1);
\draw[pre] (in1)--(hidden2);
\draw[pre] (in1)--(hidden3);
\draw[pre] (in1)--(hidden4);
\draw[pre] (in2)--(hidden5);
\draw[pre] (in2)--(hidden6);
\draw[pre] (in2)--(hidden7);
\draw[pre] (in2)--(hidden8);
\draw[post] (hidden1)--(out1);
\draw[post] (hidden2)--(out1);
\draw[post] (hidden3)--(out1);
\draw[post] (hidden4)--(out1);
\draw[post] (hidden5)--(out2);
\draw[post] (hidden6)--(out2);
\draw[post] (hidden7)--(out2);
\draw[post] (hidden8)--(out2);
\draw[peri] (out1)--(final);
\draw[peri] (out2)--(final);
\end{tikzpicture}
}
\caption{Illustration of the neural network building block from \cref{prop: multiplication_approx}.
Neurons in the input and output layers are depicted in filled dots at the top and bottom, respectively.
Applications of the activation function $\rr$ are shown as circles.
From the input values $z$ and $w$, three or two linear combinations are computed.
Then the building block from \cref{fig:approx} is inserted to approximate $z\mapsto z^2$, $z\mapsto \overline{z}^2$, or $z\mapsto z\overline{z}$.
The results are again combined linearly.\label{fig: multiplication_approx}}
\end{figure}
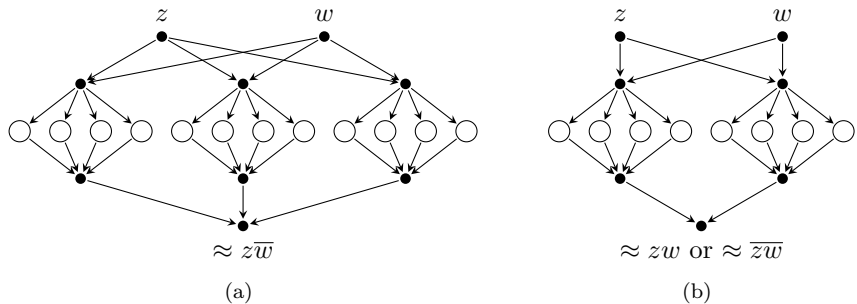
\subsection{Register model}\label{subsec:register}
In this section, we introduce the fundamental concept of the \emph{register model}.
This construction has been heavily used in \cite{KidgerLy2020} to prove the real-valued counterpart of the theorem established in the present paper.
\begin{Def}
Let $n,m,W,L\in \NN$ and $\rr:\CC \to \CC$.
  Denote by $\mathcal{I}^\rr_{n,m,W,L}$ the set of \emph{register models}
  \begin{equation*}
		T_L \circ \tilde{\rr}\circ \ldots \circ \tilde{\rr} \circ T_0,
  \end{equation*}
  where $T_0 \in \Aff(\CC^n;\CC^{W})$, $T_L \in \Aff(\CC^{W};\CC^m)$, $T_\ell \in \Aff(\CC^{W};\CC^{W})$
   and
  \begin{equation*}
    \tilde{\rr}: \CC^{W} \to \CC^{W}, \quad \lr{\tilde{\rr}(z_1,\ldots,z_{W})}_j = \begin{cases}\rr(z_{W}),& j = W,\\ z_j,&j \neq W.\end{cases}
  \end{equation*}
  In view of cases where the depth $L$ does not matter, we set 
  \begin{equation*}
  \mathcal{I}^{\rr}_{n,m,W} \defeq \bigcup_{L \in \NN} \mathcal{I}^{\rr}_{n,m,W,L}.
  \end{equation*}
\end{Def}
\begin{Bem}\label{rem:register}
The set $\mathcal{I}^{\rr}_{n,m,W,L}$ may be viewed as the set of CVNNs with $n$ input neurons, $m$ output neurons, a width of $W$ and 
a depth of $L$, where in every hidden layer the first $W-1$ neurons use the identity as activation function and in the last neuron, $\rr$ is used as activation function. 
In fact, since we can apply permutations to the entries of a vector before and after each layer, it is irrelevant in which neuron the application of $\rr$ takes place. 
We choose the last neuron for convenience. 
\end{Bem}
One can transform a shallow network into a deep narrow register model by \enquote{flipping} the shallow network and only performing one computation per layer.
This is formalized in \cref{register_model} and illustrated in \cref{fig:register}.
\begin{Prop}
    \label{register_model}
    Let $n,m,L \in \NN$.
  Let $f \in \mathcal{SN}^\rr_{n,m,L}$ and let $\tilde{f}: \CC^n \to \CC^n \times \CC^m, \ \tilde{f}(z) \defeq (z, f(z))$. 
  Then we have $\tilde{f} \in \mathcal{I}^\rr_{n,n+m,n+m+1,L}$. 
	In particular, we have $\mathcal{SN}^\rr_{n,m,L} \subseteq \mathcal{I}^\rr_{n,m,n+m+1,L}$.
\end{Prop}

\begin{proof}
Let $f\in\mathcal{SN}^\rr_{n,m}$.
Then there exist $V_1\in\Aff(\CC^n;\CC^{L})$, and $V_2\in\Aff(\CC^{L};\CC^m)$ such that $f = V_2 \circ \rr^{\times L} \circ V_1$.
For $j\in\setn{1,\ldots, m}$, the $j$th component function $f_j$ of $f$ can be written as
\begin{equation*}
f_j(z) = \lr{\sum_{k=1}^{L} c_{k,j} \rr(a_k^\top z + b_k)} + d_j
\end{equation*}
with suitably chosen $a_k \in \CC^n$ and $c_{k,j}, b_k, d_j \in \CC$.
We define
\begin{equation*}
T_0:\CC^n \to \CC^n \times \CC \times \CC^m, \quad T_1(z) = (z,  a_1^\top z + b_1, 0).
\end{equation*}
For $\ell\in\setn{1,\ldots,L-1}$, we define
\begin{equation*}
T_\ell:\CC^n \times \CC \times \CC^m \to \CC^n \times \CC \times \CC^m, \quad  T_\ell(z,u,(w_1,\ldots,w_m)) =
 \begin{pmatrix}z\\ a_{\ell+1}^\top z + b_{\ell+1} \\w_1 + c_{\ell, 1}u \\ w_2 + c_{\ell, 2} u \\ \vdots\\ w_m + c_{\ell , m}u\end{pmatrix}.
\end{equation*}
Last, we set
\begin{equation*}
T_{L}: \CC^n \times \CC \times \CC^m \to \CC^n \times \CC^m , \quad  T_{L}(z,u,(w_1,\ldots,w_m)) = 
\begin{pmatrix}z \\w_1 + c_{L, 1}u + d_1\\ w_2 + c_{L, 2} u + d_2\\ \vdots\\ w_m + c_{L, m}u + d_m\end{pmatrix}.
\end{equation*}
Then we have
\begin{equation}\label{eq:register}
\tilde{f} = T_{L} \circ \tilde{\rr} \circ \ldots \circ \tilde{\rr} \circ T_0
\end{equation}
with 
\begin{equation*}
\tilde{\rr}: \CC^{n+m+1} \to \CC^{n+m+1}, \quad \lr{\tilde{\rr}(z_1,\ldots,z_{n+m+1})}_j = \begin{cases}\rr(z_{n+1})&\text{if } j = n+1,\\ z_j&\text{if }j \neq n+1,\end{cases}
\end{equation*}
which clearly yields $\tilde{f} \in \mathcal{I}^\rr_{n,n+m,n+m+1,L}$ (see \cref{rem:register}).
The second part of the statement follows by applying a projection onto the last $m$ coordinates after the last layer. 
\end{proof}
We illustrate the proof of \cref{register_model} in the following.
To this end, let us adopt some terminology from \cite[Proof of Proposition~4.6]{KidgerLy2020}.
The neurons that use $\rr$ as the activation function (the one with index $n+1$ in each layer) shall be referred to as \emph{computation neurons}.
We call the neurons with indices $<n+1$ in each layer the \emph{in-register neurons}.
Here the inputs are just passed through the different layers unaltered.
In other words, the restriction of $\tilde{\rr} \circ \ldots \circ \tilde{\rr} \circ T_0$ in \eqref{eq:register} to the components with indices $<n+1$ is just the identity on $\CC^n$.
The remaining neurons (the ones with indices $>n+1$ in each layer) are called \emph{out-register neurons}.
Here the outputs of the computation neurons are assembled to form the final outputs.

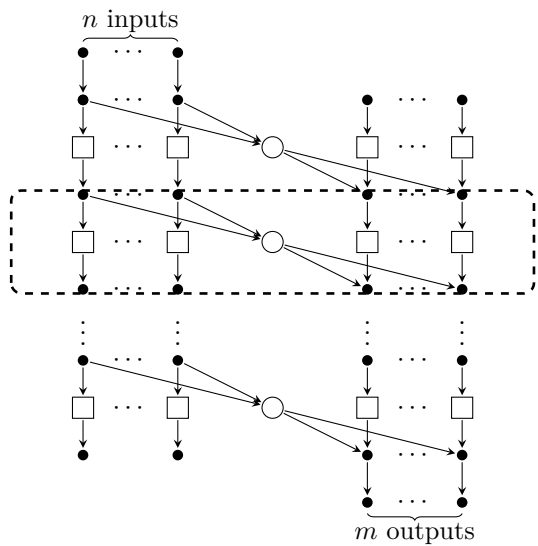
\begin{figure}[H]
\centering
\begin{tikzpicture}[x=0.9cm,y=0.63cm,>=stealth,neuron/.style={circle,draw=black,inner sep=0pt,minimum size=8pt},
transition/.style={circle,fill=black,inner sep=0pt,minimum size=4pt},
ident/.style={regular polygon sides=4,draw=black,inner sep=0pt,minimum size=8pt},
pre/.style={->, shorten >=0.5pt,shorten <=0.5pt},
post/.style={->, shorten >=0.5pt,shorten <=0.5pt},
peri/.style={->, shorten >=0.5pt,shorten <=0.5pt}]

\node[transition] (init1) at (-2.8,1){};
\node[transition] (init2) at (-1.4,1){};
\node[transition] (in1) at (-2.8,0){};
\node[transition] (in2) at (-1.4,0){};
\node[transition] (out1) at (-2.8,0){};
\node[transition] (in3) at (1.4,0){};
\node[transition] (in4) at (2.8,0){};

\node[ident] (ir11) at (-2.8,-1){};
\node[ident] (ir12) at (-1.4,-1){};
\node[neuron] (c1) at (0,-1){};
\node[ident] (or11) at (1.4,-1){};
\node[ident] (or12) at (2.8,-1){};

\node[transition] (ir21) at (-2.8,-2){};
\node[transition] (ir22) at (-1.4,-2){};
\node[transition] (or21) at (1.4,-2){};
\node[transition] (or22) at (2.8,-2){};

\node[ident] (ir31) at (-2.8,-3){};
\node[ident] (ir32) at (-1.4,-3){};
\node[neuron] (c3) at (0,-3){};
\node[ident] (or31) at (1.4,-3){};
\node[ident] (or32) at (2.8,-3){};

\node[transition] (ir41) at (-2.8,-4){};
\node[transition] (ir42) at (-1.4,-4){};
\node[transition] (or41) at (1.4,-4){};
\node[transition] (or42) at (2.8,-4){};

\node[transition] (ir51) at (-2.8,-5.5){};
\node[transition] (ir52) at (-1.4,-5.5){};
\node[transition] (or51) at (1.4,-5.5){};
\node[transition] (or52) at (2.8,-5.5){};

\node[ident] (ir61) at (-2.8,-6.5){};
\node[ident] (ir62) at (-1.4,-6.5){};
\node[neuron] (c6) at (0,-6.5){};
\node[ident] (or61) at (1.4,-6.5){};
\node[ident] (or62) at (2.8,-6.5){};

\node[transition] (ir71) at (-2.8,-7.5){};
\node[transition] (ir72) at (-1.4,-7.5){};
\node[transition] (or71) at (1.4,-7.5){};
\node[transition] (or72) at (2.8,-7.5){};

\node[transition] (out1) at (-2.8,-8.5){};
\node[transition] (out2) at (-1.4,-8.5){};
\node[transition] (out4) at (1.4,-8.5){};
\node[transition] (out5) at (2.8,-8.5){};

\draw ($(init1)!0.5!(init2)$) node{$\cdots$};
\draw ($(in1)!0.5!(in2)$) node{$\cdots$};
\draw ($(in3)!0.5!(in4)$) node{$\cdots$};
\draw ($(ir11)!0.5!(ir12)$) node{$\cdots$};
\draw ($(ir21)!0.5!(ir22)$) node{$\cdots$};
\draw ($(ir31)!0.5!(ir32)$) node{$\cdots$};
\draw ($(ir41)!0.5!(ir42)$) node{$\cdots$};
\draw ($(ir51)!0.5!(ir52)$) node{$\cdots$};
\draw ($(ir61)!0.5!(ir62)$) node{$\cdots$};
\draw ($(or11)!0.5!(or12)$) node{$\cdots$};
\draw ($(or21)!0.5!(or22)$) node{$\cdots$};
\draw ($(or31)!0.5!(or32)$) node{$\cdots$};
\draw ($(or41)!0.5!(or42)$) node{$\cdots$};
\draw ($(or51)!0.5!(or52)$) node{$\cdots$};
\draw ($(or61)!0.5!(or62)$) node{$\cdots$};
\draw ($(or71)!0.5!(or72)$) node{$\cdots$};
\draw ($(out4)!0.5!(out5)$) node{$\cdots$};
\draw ($(ir41)!0.5!(ir51)$) node{$\vdots$};
\draw ($(ir42)!0.5!(ir52)$) node{$\vdots$};
\draw ($(or41)!0.5!(or51)$) node{$\vdots$};
\draw ($(or42)!0.5!(or52)$) node{$\vdots$};

 \draw [decoration={brace,raise=2pt},decorate] (init1.north) -- (init2.north) node [above=2pt,pos=0.5] {$n$ inputs};
 \draw [decoration={brace,mirror,raise=2pt},decorate] (out4.south) -- (out5.south) node [below=2pt,pos=0.5] {$m$ outputs};
  \draw [decoration={brace,mirror,raise=2pt},decorate] (out1.south) -- (out2.south) node [below=2pt,pos=0.5] {$n$ outputs};

\draw[peri] (init1)--(in1);
\draw[peri] (init2)--(in2);
\draw[pre] (in1)--(ir11);
\draw[pre] (in2)--(ir12);
\draw[pre] (in3)--(or11);
\draw[pre] (in4)--(or12);
\draw[post] (ir11)--(ir21);
\draw[post] (ir12)--(ir22);
\draw[pre] (ir21)--(ir31);
\draw[pre] (ir22)--(ir32);
\draw[post] (ir31)--(ir41);
\draw[post] (ir32)--(ir42);
\draw[pre] (ir51)--(ir61);
\draw[pre] (ir52)--(ir62);

\draw[post] (in1)--(c1);
\draw[post] (in2)--(c1);
\draw[post] (ir21)--(c3);
\draw[post] (ir22)--(c3);
\draw[post] (ir51)--(c6);
\draw[post] (ir52)--(c6);

\draw[pre] (c1)--(or21);
\draw[pre] (c1)--(or22);
\draw[pre] (c3)--(or41);
\draw[pre] (c3)--(or42);
\draw[pre] (c6)--(or71);
\draw[pre] (c6)--(or72);

\draw[post] (or11)--(or21);
\draw[post] (or12)--(or22);
\draw[pre] (or21)--(or31);
\draw[pre] (or22)--(or32);
\draw[post] (or31)--(or41);
\draw[post] (or32)--(or42);
\draw[pre] (or51)--(or61);
\draw[pre] (or52)--(or62);
\draw[post] (ir61)--(ir71);
\draw[post] (ir62)--(ir72);
\draw[post] (or61)--(or71);
\draw[post] (or62)--(or72);

\draw[peri] (ir71)--(out1);
\draw[peri] (ir72)--(out2);
\draw[peri] (or71)--(out4);
\draw[peri] (or72)--(out5);

\draw[rounded corners,dashed,line width=1pt] ($(ir41.south west)+(-1,0)$) rectangle ($(or22.north east)+(1,0)$);

\end{tikzpicture}
\caption{Illustration of the register model from \cref{register_model}.
Neurons where the complex identity is used as activation function are visualized as squares, whereas neurons using $\rr$ as activation function are visualized using circles.
The in-register neurons (squares on the left) store the input values and pass them to the computation neurons (middle circles).
The result of the computations are added up and stored in the out-register neurons (squares on the right).
The dashed box highlights one of the blocks that are later replaced using approximations of the complex identity.\label{fig:register}}
\end{figure}

The following proposition enables us to approximate CVNNs that use $\overline{\rr}$ as activation function locally uniformly by CVNNs that use $\rr$ 
as activation function, if there exists a point $z_0 \in \CC$ with $\wirt \rr(z_0) = 0 \neq \wirtq \rr(z_0)$.

\begin{Prop}\label{prop:conj_universal}
Let $\rr \in C(\CC;\CC)$ and $n,m,W,L \in \NN$.
Assume that there exists $z_0 \in \CC$ such that $\rr$ is real differentiable at $z_0$ with 
\begin{equation*}
\wirt \rr(z_0) = 0 \quad \text{and} \quad \wirtq \rr(z_0) \neq 0.
\end{equation*}
Then 
\begin{equation*}
\mathcal{NN}^{\overline{\rr}}_{n,m,W,L} \subseteq \overline{\mathcal{NN}^{\rr}_{n,m,W,2L}},
\end{equation*}
where the closure is taken with respect to the compact-open topology.
\end{Prop}
\begin{proof}
 Let $f \in \mathcal{NN}^{\overline{\rr}}_{n,m,W}$ be arbitrary and consider the decomposition
	\begin{equation*}
	f = V_L \circ \overline{\rr}^{\times W} \circ V_{L-1} \circ \ldots \circ \overline{\rr}^{\times W} \circ V_0
	\end{equation*}
	with $V_0 \in \Aff(\CC^n; \CC^{W})$, $V_\ell \in \Aff(\CC^{W}; \CC^{W})$, and $V_L \in \Aff(\CC^{W}; \CC^m)$ for every $\ell\in\setn{1,\ldots,L-1}$.
	From \Cref{prop: main}\ref{prop: main item 2} and \cref{prop:equivalence} we infer the existence of sequences $(\Phi_j)_{j \in \NN}$ and $(\Psi_j)_{j \in \NN}$ of 
	affine maps $\Phi_j ,\Psi_j\in \Aff(\CC^{W}; \CC^{W})$ for $j \in \NN$ such that
	\begin{equation*}
	\Psi_j \circ \rr^{\times W} \circ \Phi_j \xrightarrow{j\to \infty} \overline{\id_\CC}^{\times W}
	\end{equation*}
	locally uniformly.
	But then we see
	\begin{equation*}
	V_L \circ \lr{ \Psi_j \circ \rr^{\times W} \circ \Phi_j} \circ \rr^{\times W}\circ V_{L-1} \circ \ldots \circ \lr{ \Psi_j \circ \rr^{\times W} \circ \Phi_j} \circ \rr^{\times W}\circ V_0 \xrightarrow{j\to \infty}f
	\end{equation*}
	locally uniformly, where the left-hand side is an element of $\mathcal{NN}^\rr_{n,m,W,2L}$ for every $j \in \NN$.
	Here we applied \Cref{prop:comp_loc_conv}.
	The claim then follows from \cref{prop:comp_open_loc_unif_conv}.
\end{proof}
Using the previous proposition, we can now show that every register model of width $W$ can be approximated by CVNNs, where in \emph{every} hidden neuron the function $\rr$ is used as activation function. 
To do so, it is necessary to approximate the identity connections that appear in the register model.
By assumption, the activation function $\rr$ is real differentiable at some point $z_0\in\CC$ with non-zero derivative.
We consider three different cases.
First, if $\rr$ is even \emph{complex} differentiable at $z_0$, \cref{prop: main} yields an approximation of $\id_\CC$ using shallow CVNNs with activation function $\rr$ and a width of $1$.
In that case, replacing the identity connections in the register model by these approximations results in a sufficient width of $W$.
Second, if $\overline{\rr}$ is complex differentiable at $z_0$, we use \cref{prop:conj_universal} to get the desired result. 
Last, if neither of the two former cases happens, we show that a width of $2W-1$ is sufficient for CVNNs with activation function $\rr$ in order to be universal, by using the third building block described in \cref{prop: main}.
Note that in each layer, there are $W-1$ identity connections to be replaced (resulting in a width of $2(W-1)= 2W-2$ for the identity connections) and one additional application of the activation function $\rr$, which in total gives us a width of $2W-1$.
\begin{Prop}\label{prop:register_to_full}
Let $\varrho \in C(\CC;\CC)$. 
Assume that there exists a point $z_0 \in \CC$ such that $\rr$ is real differentiable at $z_0$ with non-zero derivative. 
Let $n,m,W,L \in \NN$.
	\begin{enumerate}[label={(\roman*)},leftmargin=*,align=left,noitemsep]
		\item{\label{item:1}
		If $\wirt \rr(z_0) \neq 0$ and $\wirtq\rr(z_0) = 0$, then $\mathcal{I}^\rr_{n,m,W,L} \subseteq \overline{\mathcal{NN}^\rr_{n,m,W,L}}$.}
		\item{\label{item:2}
		If $\wirt \rr(z_0) = 0$ and $\wirtq\rr(z_0) \neq 0$, then $\mathcal{I}^\rr_{n,m,W,L} \subseteq \overline{\mathcal{NN}^\rr_{n,m,W,2L}}$.}
		\item{\label{item:3}
		If $\wirt \rr(z_0) \neq 0$ and $\wirtq\rr(z_0) \neq 0$, then $\mathcal{I}^\rr_{n,m,W,L} \subseteq \overline{\mathcal{NN}^\rr_{n,m,2W - 1,L}}$.}
	\end{enumerate}
	Here, the closure is taken with respect to the compact-open topology. 
\end{Prop}
\begin{proof}
Let $g \in \mathcal{I}^{\rr}_{n,m,W,L}$.
This means that there exist maps $T_0 \in \Aff(\CC^n;\CC^{W})$, $T_L \in \Aff(\CC^{W};\CC^m)$, and $T_\ell \in \Aff(\CC^{W};\CC^{W})$ for $\ell\in\setn{1,\ldots,L-1}$ such that
	\begin{equation*}
		g = T_L \circ \tilde{\rr} \circ T_{L-1} \circ \ldots \circ \tilde{\rr} \circ T_0,
	\end{equation*}
	where $\tilde{\rr}:\CC^{W} \to \CC^{W}$ is given by
	\begin{equation*}
		\lr{\tilde{\rr}(z_1,\ldots,z_{W})}_j = \begin{cases}z_j&\text{if }j \in\setn{1,\ldots, W-1}, \\ \rr(z_{W})&\text{if } j =W.\end{cases}
	\end{equation*}
	We first prove \ref{item:1}.
	Combining \cref{prop: main}\ref{prop: main item 1} and \cref{prop:equivalence}, we deduce the existence of sequences $(\Psi_j)_{j \in \NN}$ and $(\Phi_j)_{j \in \NN}$ with $\Psi_j, \Phi_j \in \Aff(\CC^{W};\CC^{W})$ for $j\in\NN$, satisfying
	\begin{equation*}
		\Psi_j \circ \rr^{\times W} \circ \Phi_j
		\xrightarrow{j\to \infty} \tilde{\rr}
	\end{equation*}
	locally uniformly.
	\cref{prop:comp_loc_conv} now implies that the sequence $(g_j)_{j\in\NN}$ given by
	\begin{align*}
		g_j = T_L \circ \lr{\Psi_j \circ \rr^{\times W} \circ \Phi_j} \circ T_{L-1} \circ \ldots \circ  \lr{\Psi_j \circ \rr^{\times W} \circ \Phi_j} \circ T_0
	\end{align*}
	converges locally uniformly to $g$ as $j\to\infty$.
	Moreover, since the composition of affine maps is affine, we have $g_j \in \mathcal{NN}^\rr_{n,m,W,L}$ for all $j\in\NN$.
	Claim \ref{item:1} now follows from \cref{prop:comp_open_loc_unif_conv}.
	
	 We now deal with \ref{item:2}.
	By the fundamental properties of Wirtinger derivatives (see for instance \cite[E.~1a]{kaup_holomorphic_1983}) we compute 
	\begin{equation*}
	\wirt \overline{\rr}(z_0) = \overline{\wirtq \rr (z_0)} \neq 0 \quad\text{and}\quad \wirtq \overline{\rr}(z_0) = \overline{\wirt \rr (z_0)} = 0.
	\end{equation*}
	From \ref{item:1} we infer $\mathcal{I}^\rr_{n,m,W,L} \subseteq \overline{\mathcal{NN}^{\overline{\rr}}_{n,m,W,L}}$.
	A direct application of \cref{prop:conj_universal} yields \ref{item:2}.

	The proof of \ref{item:3} is analogous to the proof of \ref{item:1}.
	Gluing together one copy of the complex identity $\id_\CC$ and $W-1$ copies of the function $\phi$ or 
	(the projection onto the first component of) $\psi$ of \cref{prop: main}\ref{prop: main item 3}, respectively,
	 we construct sequences $(\Psi_j)_{j \in \NN},(\Phi_j)_{j \in \NN}$ satisfying $\Psi_j\in\Aff(\CC^{2W-1};\CC^{W})$
	  and $\Phi_j\in\Aff(\CC^{W};\CC^{2W-1})$ for $j\in\NN$ such that
	\begin{equation*}
		\Psi_j \circ \rr^{\times (2W - 1)} \circ \Phi_j
		\xrightarrow{j\to \infty}\tilde{\rr}
	\end{equation*}
	locally uniformly.
	Because each of the building blocks consists of $2$ neurons in this case,
	 the resulting approximating neural network has width $2W - 1$ instead of $W$.
\end{proof}

\section{Proof of the main result}\label{sec:main}
In this section, we prove \cref{thm:intro}.
The analysis is split into the case of non-polyharmonic activation functions (see \cref{chap:non-polyharmonic}) and polyharmonic activation functions (see \cref{chap:polyharmonic}).
\subsection{The non-polyharmonic case}
\label{chap:non-polyharmonic}
When the activation function $\rr$ is not polyharmonic, the universal approximation theorem for shallow CVNNs from \cite[Theorem~1.3]{Voigtlaender2022} is applicable.
For convenience, we state the following special case relevant for our investigations.
\begin{Satz}
\label{thm:felix}
Let $n\in\NN$ and $\rr \in C(\CC;\CC)$.
Then $\mathcal{SN}^\rr_{n,1}$ is universal if and only if $\rr$ is not polyharmonic.
\end{Satz}
Since local uniform convergence on $\CC^m$ is equivalent to componentwise local uniform convergence, we conclude that, 
if $\rr$ is non-polyharmonic, the set $\mathcal{SN}^\rr_{n,m}$ is universal for every $m \in \NN$.

Since the set of shallow networks with non-polyharmonic activation function 
is universal and each shallow network can be represented by a suitable register model (see \cref{register_model}), which in turn can be locally uniformly
approximated by deep narrow CVNNs (see \cref{prop:register_to_full}), we get the following result.  
\begin{Satz}\label{main_theorem_non_poly_classical_reg}
	Let $n,m \in \NN$.
	Assume that $\rr \in C(\CC;\CC)$ is not polyharmonic and that there exists $z_0 \in \CC$ such that $\rr$ is real differentiable at $z_0$ with non-zero derivative.
	\begin{enumerate}[label={(\roman*)},leftmargin=*,align=left,noitemsep]
		\item{\label{main_theorem_non_poly_classical_reg_item1}
		If $\wirt \rr(z_0) \neq 0$ and $\wirtq\rr(z_0) = 0$, then the set $\mathcal{NN}^\rr_{n,m,n+m+1}$ is universal.}
		\item{\label{main_theorem_non_poly_classical_reg_item11}
		If $\wirt \rr(z_0) = 0$ and $\wirtq\rr(z_0) \neq 0$, then the set $\mathcal{NN}^\rr_{n,m,n+m+1}$ is universal.}
		\item{\label{main_theorem_non_poly_classical_reg_item2}
		If $\wirt \rr(z_0) \neq 0$ and $\wirtq\rr(z_0) \neq 0$, then the set $\mathcal{NN}^\rr_{n,m,2n+2m+1}$ is universal.}
	\end{enumerate}

\end{Satz}
\begin{proof}
	Note that 
	\begin{equation*}
	\mathcal{SN}^{\rr}_{n,m} \subseteq \mathcal{I}^{\rr}_{n,m,n+m+1}
	\end{equation*}
	according to the second part of \cref{register_model}.
	Since $\rr$ is non-polyharmonic, $\mathcal{SN}^{\rr}_{n,m}$ is universal according to \cref{thm:felix}
	and we hence have
	\begin{equation*}
	C(\CC^n; \CC^m) = \overline{\mathcal{SN}^\rr_{n,m}} \subseteq \overline{\mathcal{I}^\rr_{n,m,n+m+1}}.
	\end{equation*}
	The claim then follows from \cref{prop:register_to_full,prop:equi}.
\end{proof}

Next we provide two examples of activation functions which are used in practice and to which \cref{main_theorem_non_poly_classical_reg} applies.
\begin{Bsp}
The modReLU function has for example been proposed in \cite{arjovsky2016unitary} as a generalization of the classical ReLU to the complex plane.
For a parameter $b< 0$ it is defined as
\begin{equation*}
        \modrelu_b: \CC \to \CC, \quad \modrelu_b(z) \defeq \begin{cases}(\abs{z} + b)\frac{z}{\abs{z}}&\text{if }\abs{z} + b \geq 0, \\ 0& \text{otherwise.}\end{cases}
    \end{equation*}
An application of \cref{main_theorem_non_poly_classical_reg}\ref{main_theorem_non_poly_classical_reg_item2} shows that for $n,m \in \NN$, and $b<0$, the set $\mathcal{NN}^{\modrelu_b}_{n,m,2n+2m+1}$ is universal.

To this end, let us verify the assumptions of \cref{main_theorem_non_poly_classical_reg} in detail.
Since the continuity of $\modrelu_b$ is immediate for $z\in\CC$ with $\abs{z} \neq -b$, it remains to check the case $\abs{z} = -b$.
Take any sequence $(z_j)_{j\in \NN}$ with $z_j \to z$ as $j\to\infty$, where we assume without loss of generality $\abs{z_j} \geq -b$ for every $j \in \NN$.
Then $\abs{\modrelu_b(z_j)} = \abs{z_j} + b \to \abs{z} + b = 0$ as $j\to\infty$.
This shows the continuity of $\modrelu_b$.

In \cite[Corollary~5.4]{Geuchen2023} it is shown that for all $z \in \CC$ with $\abs{z} > -b$ and all $k, \ell \in \NN_0$ one has
\begin{equation*}
\wirt^k \wirtq^\ell \modrelu_b(z) \neq 0.
\end{equation*}
This implies that $\modrelu_b$ is non-polyharmonic and $\wirt \modrelu_b (z) \neq 0 \neq \wirtq \modrelu_b (z)$ for all $z \in \CC$ with $\abs{z} > -b$.

Further note that the result from \cref{main_theorem_non_poly_classical_reg} cannot be used to reduce the sufficient width to $n+m+1$ 
since it holds $\wirt \modrelu_b (z) \neq 0 \neq \wirtq \modrelu_b (z)$ for all $z \in \CC$ with $\abs{z} > -b$, 
$\wirt \modrelu_b(z) = \wirtq \modrelu_b(z) = 0$ for every $z \in \CC$ with $\abs{z} < -b$ and $\modrelu_b$ is not real differentiable at any $z \in \CC$ with $\abs{z} = -b$.
\end{Bsp}

\begin{Bsp}
The complex cardioid function has been used in \cite{virtue2017better} in the context of MRI fingerprinting, where complex-valued neural networks significantly outperformed their real-valued counterparts.
It is defined as
\begin{equation*}
        \card: \CC \to \CC, \quad \card(z) \defeq \begin{cases}\frac{1}{2}\lr{1+\frac{\RE(z)}{\abs{z}}}z&\text{if }z\in\CC \setminus \setn{0},\\0&\text{if }z=0.\end{cases}
    \end{equation*}
An application of \cref{main_theorem_non_poly_classical_reg}\ref{main_theorem_non_poly_classical_reg_item1} shows that for  $n,m \in \NN$ the set $\mathcal{NN}^{\card}_{n,m,n+m+1}$ is universal.

To this end, let us verify the assumptions of \cref{main_theorem_non_poly_classical_reg} in detail.
The continuity of $\card$ on $\CC \setminus \setn{0}$ is immediate.
Further note that
\begin{equation*}
\abs{\card(z)} = \abs{\frac{1}{2}\lr{1+\frac{\RE(z)}{\abs{z}}}z} \leq \abs{z} \to 0
\end{equation*}
as $z\to 0$, showing the continuity of $\card$ on the entire complex plane $\CC$.
Now \cite[Corollaries~5.6 and~5.7]{Geuchen2023} show that $\card$ is non-polyharmonic and
\begin{equation*}
\wirt \card (z) = \frac{1}{2} + \frac{1}{8} \cdot \frac{\overline{z}}{\abs{z}} + \frac{3}{8} \cdot \frac{z}{\abs{z}}, \quad \wirtq \card (z) = - \frac{1}{8} \cdot \frac{z^3}{\abs{z}^3} + \frac{1}{8} \cdot \frac{z}{\abs{z}}
\end{equation*}
for every $z \in\CC \setminus \setn{0}$.
Hence, we see $\wirt \card (1) = 1$ and $\wirtq \card (1) = 0$.
\end{Bsp}

\subsection{The polyharmonic case}\label{chap:polyharmonic}

In this section, we deal with activation functions $\rr: \CC \to \CC$ that are polyharmonic.
However, it turns out that this property can be relaxed to only requiring that $\rr \in C^2(\CC;\CC)$ in order for the proofs to work.
Note that we still assume that the activation function $\rr$ is neither holomorphic, nor antiholomorphic, nor $\RR$-affine.
That these assumptions cannot be neglected is shown in \cref{thm:holo_anti_aff}.
The main assumptions of this subsection can therefore be stated as follows.
\begin{Assum}
\label{assumptions}
Let $\rr: \CC \to \CC$ be a function satisfying the following conditions:
\begin{enumerate}[label={(\roman*)},leftmargin=*,align=left,noitemsep]
\item{$\rr\in C^2(\CC;\CC)$,}
\item{$\rr$ is not holomorphic,}
\item{$\rr$ is not antiholomorphic,}
\item{$\rr$ is not $\RR$-affine.}
\end{enumerate}
\end{Assum}
Using \cref{prop: main}, we derive the following \cref{prop: approx_identity_conj}.
It states that the function $z \mapsto (z, \overline{z})$ can be uniformly approximated on compact sets by a shallow network of width $2$.
\begin{Prop} \label{prop: approx_identity_conj}
Let $\rr$ satisfy \cref{assumptions}, let $K \subseteq \CC$ be compact, and $\eps > 0$.
Then there exist maps $\phi \in \Aff(\CC;\CC^2)$ and $\psi \in \Aff(\CC^2;\CC^2)$ such that
\begin{equation*}
\sup_{z\in K}\mnorm{(\psi \circ \rr^{\times 2} \circ \phi)(z) - (z, \overline{z})}_{\CC^2} < \eps.
\end{equation*}
\end{Prop}
\begin{proof}
If there exists a point $z_0 \in \CC$ with $\wirt \rr (z_0) \neq 0 \neq \wirtq \rr(z_0)$, we can directly apply \cref{prop: main}\ref{prop: main item 3}.
If there does not exist such a point $z_0$, we can still find $z_1 \in \CC$ with $\wirt \rr (z_1) \neq 0$ and $z_2 \in \CC$ with $\wirtq \rr (z_2) \neq 0$, since $\rr$ is neither holomorphic nor antiholomorphic.
By assumption of the nonexistence of $z_0 \in \CC$ with $\wirt \rr (z_0) \neq 0 \neq \wirtq \rr(z_0)$, it follows that $\wirtq \rr (z_1) = 0 $ and $\wirt \rr (z_2)  = 0$, and we can thus apply \cref{prop: main}\ref{prop: main item 1} and \ref{prop: main item 2}.
\end{proof}
\cref{prop: approx_identity_conj} is a central finding, since it implies that we can build from the activation function $\rr$ constructions
 that are similar to the register model construction from \cref{register_model}.

The Stone--Weierstrass theorem states that any continuous function can be arbitrarily well approximated by complex 
polynomials in $z_1,\ldots,z_n$ and $\overline{z_1},\ldots,\overline{z_n}$ in the uniform norm on compact subsets of $\CC^n$.
In order to show universality of CVNNs, it therefore suffices to show that such polynomials can be approximated by deep narrow 
CVNNs to arbitrary precision in the uniform norm on compact sets.
Motivated by this observation, for $N,n \in \NN$, we define 
\begin{equation}\label{eq:poly}
\mathcal{P}_N^n \defeq \setcond{p:\CC^n \to \CC, \ p(z)= \underset{\m \leq N}{\sum_{\m \in \NN_0^n}} \underset{\elll \leq N}{\sum_{\elll \in \NN_0^n}}  
a_{\m, \elll} z ^{\m} \overline{z}^{\elll}}{a_{\m, \elll}\in\CC\fall \m,\elll}
\end{equation}
as the set of complex polynomials on $\CC^n$ of degree less than $N$ in each variable.
Here, the notation $\m \leq N$ means $\m_j \leq N$ for every $j \in \setn{1,\ldots,n}$ and 
\begin{equation*}
z^{\m}z^{\elll} \defeq \prod_{j=1}^n z_j^{\m_j} \overline{z_j}^{\elll_j}
\end{equation*}
for $\m=(\m_1,\ldots,\m_n), \elll=(\elll_1,\ldots,\elll_n)\in\NN_0^n$.
We follow an approach similar to that of the register model by preserving the inputs and outputs from layer to layer while 
gradually performing multiplications to approximate the individual monomials.
This is described in the following lemma.
In its statement and proof, we use $\prod_{k=1}^n f_k^{\beta_k}$ as an abbreviation for the composition
\begin{equation*}
\underbrace{(f_n\circ\ldots\circ f_n)}_{\beta_n\text{ many}}\circ\ldots\circ \underbrace{(f_2\circ\ldots\circ f_2)}_{\beta_2\text{ many}}\circ \underbrace{(f_1\circ\ldots\circ f_1)}_{\beta_1\text{ many}}.
\end{equation*}
\begin{Lem}\label{lem:thomas_model}
Let $p:\CC^n\to\CC^m$, $p(z)=(p_1(z),\ldots,p_m(z))$ such that $p_j \in \mathcal{P}^n_N$ for every $j \in \setn{1,\ldots,n}$
for a suitable choice of $N \in \NN$.
Let $\mul$ be one of the three maps
\begin{align*}
\mul_1:\CC\times \CC\to\CC,\qquad &\mul_1(z,w)=zw,\\
\mul_2:\CC\times \CC\to\CC,\qquad &\mul_2(z,w)=z\overline{w},\\
\mul_3:\CC\times \CC\to\CC,\qquad &\mul_3(z,w)=\overline{z}\overline{w},
\end{align*}
as in \cref{prop: mul_approx}.
Further, let 
\begin{equation*}
f_k^{\mul} = f_k:\CC^n\times\CC\times\CC^m\to\CC^n\times\CC\times\CC^m,\qquad f_k(z,w,u)=(z,\mul(w,z_k),u) 
\end{equation*}
for $k\in\setn{1,\ldots,n}$, and 
\begin{equation*}
h_k:\CC^{n+1+m}\to\CC^{n+1+m},\qquad h_k=\id_{\CC^{k-1}}\times\overline{\id_\CC}\times\id_{\CC^{n+m-k}},
\end{equation*}
for $k\in\setn{1,\ldots,n+1}$.
Then there exist $L \in \NN$, maps $g_j:\CC^{n+m+1}\to\CC^{n+m+1}$ for $j\in\setn{1,\ldots,L}$
each of which is a finite composition of the maps $f_1,\ldots,f_n,h_1,\ldots,h_{n+1}$ defined above, and affine mappings $T_0 \in \Aff(\CC^n; \CC^{n+m+1})$, 
$T_\ell \in \Aff(\CC^{n+m+1}; \CC^{n+m+1})$ for $\ell \in \setn{1,\ldots,L-1}$ and further 
$T_L \in \Aff(\CC^{n+m+1}; \CC^m)$ such that
\begin{equation*}
p=T_L\circ g_L\circ \ldots T_1\circ g_1\circ T_0.
\end{equation*}
\end{Lem}
\begin{proof}
By definition of $\mathcal{P}_N^n$, there exist $a_{j,\m, \elll} \in \CC$ for $\m,\elll \in \NN_0^n$ with $\m, \elll \leq N$ and $j \in \setn{1,\ldots,m}$ such that 
\begin{equation*}
p_j(z) =\underset{\m, \elll \leq N}{\sum_{\m,\elll \in \NN_0^n}}a_{j,\m, \elll}z^{\m}\overline{z}^{\elll} \quad \text{for every } j \in \setn{1,\ldots,m}.
\end{equation*}
By enumerating all the occurring monomials from 1 to $L$, we can rewrite 
\begin{equation*}
p_j(z) = \sum_{\ell = 1}^L c_{j,\ell} \prod_{k=1}^n z_k^{\alpha_{\ell,k}} \overline{z_k}^{\beta_{\ell,k}} 
\end{equation*}
with $\alpha_{\ell,k}, \beta_{\ell,k} \in \setn{0,\ldots,N}$ and $c_{j,\ell} \in \CC$ for every $\ell \in \setn{1,\ldots,L}$, $k \in \setn{1,\ldots,n}$, and $j \in \setn{1,\ldots,m}$.

We then set 
\begin{align*}
&T_0\colon\CC^n\to \CC^n\times \CC\times\CC^m,\quad &T_0(z)&=(z,1,0),\\
&T_\ell\colon\CC^n\times\CC\times\CC^m\to\CC^n\times\CC\times \CC^m,\quad &T_\ell(z,w,u)&=\lr{z,1,w \sum_{j=1}^mc_{j,\ell}e_j+u}, \\ \intertext{and}
&T_{L}\colon\CC^n\times\CC\times\CC^m \to  \CC^m,\quad &T_{L}(z,w,u)&=w \sum_{j=1}^mc_{j,L}e_j+u,
\end{align*}
where $\ell \in \setn{1,\ldots,L-1}$.
Here, $e_j$ denotes the $j$th standard basis vector in $\CC^m$.
Clearly, the maps $T_0,\ldots,T_L$ are $\CC$-affine. 
Moreover, for $\ell\in\setn{1,\ldots,L}$, let
\begin{equation*}
g_\ell=\begin{cases}\prod_{k=1}^n f_k^{\alpha_{\ell,k}}\circ h_{n+1} \circ \prod_{k=1}^n f_k^{\beta_{\ell,k}}&\text{if }\mul=\mul_1,\\
\prod_{k=1}^n f_k^{\beta_{\ell,k}}\circ h_{n+1}\circ \prod_{k=1}^n f_k^{\alpha_{\ell,k}}&\text{if }\mul=\mul_2,\\
\prod_{k=1}^n(h_{n+1}\circ f_k)^{\alpha_{\ell,k}}\circ \prod_{k=1}^n (h_{n+1}\circ h_k\circ f_k\circ h_k)^{\beta_{\ell,k}}&\text{if }\mul=\mul_3.
\end{cases}
\end{equation*}
Then
\begin{equation}\label{eq:monomial-assembly}
g_\ell(z,1,u)=\lr{z,\prod_{k=1}^n z_k^{\alpha_{\ell,k}}\overline{z_k}^{\beta_{\ell,k}},u}
\end{equation}
for all $\ell\in\setn{1,\ldots,L}$, $z\in\CC^n$, and $u\in\CC$, i.e., the $(n+1)$-entry of $g_\ell(z,1,u)$ is the $\ell$th monomial.
The application of $T_\ell$ adds this monomial with the correct coefficients to the cumulative sums 
in the entries indexed $n+2,\ldots,n+1+m$, and resets the $(n+1)$-entry to $1$ for the assembly of the next monomial. 
This proves 
\begin{equation*}
p = T_L \circ g_L \circ \ldots \circ T_1 \circ g_1 \circ T_0.
\end{equation*}

To see \eqref{eq:monomial-assembly}, note that for example 
\begin{align*}
&(h_{n+1}\circ h_k\circ f_k\circ h_k)(z_1,\ldots,z_k,\ldots,z_n,w,u)\\
&=(h_{n+1}\circ h_k\circ f_k)(z_1,\ldots,\overline{z_k},\ldots,z_n,w,u)\\
&=(h_{n+1}\circ h_k)(z_1,\ldots,\overline{z_k},\ldots,z_n,z_k\overline{w},u)\\
&=h_{n+1}(z_1,\ldots,z_k,\ldots,z_n,z_k\overline{w},u)\\
&=(z_1,\ldots,z_k,\ldots,z_n,\overline{z_k}w,u)
\end{align*}
when $\mul=\mul_3$.
\end{proof}
\cref{lem:thomas_model} states that every function from $\CC^n$ to $\CC^m$ whose components are polynomials can be written as the composition of affine maps 
and the maps $f_k$ and $h_k$ that are defined in \cref{lem:thomas_model}.
We thus aim to show that each of the $f_k$ and $h_k$ can be locally uniformly approximated by narrow CVNNs. 
For $h_k$, this is a direct consequence of \cref{prop: approx_identity_conj}.
For $f_k$, since each function $\mul$ can be approximated by a shallow CVNN of width 12 (see \cref{prop: mul_approx}), together with \cref{prop: approx_identity_conj} 
it is easy to see that $f_k \in \overline{\mathcal{SN}^\rr_{n+m+1,n+m+1, 2n+2m +12}}$. 
However, a careful analysis shows that $f_k \in \overline{\mathcal{NN}^\rr_{n+m+1,n+m+1, 2n+2m +5}}$; see the proof of \cref{thm:main-polyharmonic}.
The next proposition is a crucial ingredient for this and shows that each $f_k$ can be approximated by register models with a width of $n+m+3$.
\begin{Prop}\label{prop:fk}
Let $\varrho \in C^2(\CC; \CC)$ be not $\RR$-affine, $n,m \in \NN$, and $k \in \setn{1,\ldots,n}$.
Let further $\mul $ be chosen according to \cref{prop: multiplication_approx} and $f_k^{\mul} = f_k : \CC^{n+m+1} \to \CC^{n+m+1}$ be defined as in \Cref{lem:thomas_model}.
Then we have 
\begin{equation*}
f_k \in \overline{\mathcal{I}^{\varrho}_{n+m+1,n+m+1,n+m+3}},
\end{equation*}
where the closure is taken with respect to the compact-open topology. 
\end{Prop}
\begin{proof}
According to \Cref{prop: mul_approx} and \Cref{prop:equivalence}, we may pick a sequence $(\varphi_j)_{j \in \NN}$ with $\varphi_j \in \mathcal{SN}^\varrho_{2,1}$ for $j \in \NN$ that satisfies 
\begin{equation*}
\varphi_j \to  \mul \quad \text{locally uniformly.}
\end{equation*}
We define 
\begin{equation*}
\eta_j: \quad \CC^n \times \CC \times \CC^m \to \CC^n \times \CC \times \CC^m, \quad \eta_j(z,w,u) \defeq (z, \varphi_j(w,z_k),u)
\end{equation*}
and note that 
\begin{equation*}
\eta_j \to f_k \quad \text{locally uniformly.}
\end{equation*}
Hence, if we can show that $\eta_j \in \mathcal{I}^{\varrho}_{n+m+1,n+m+1,n+m+3}$ for every $j \in \NN$, the proof is completed according to \cref{prop:comp_open_loc_unif_conv}. 

To this end, we take $j \in \setn{1,\ldots,n}$ and define 
\begin{align*}
\gamma_j &\colon \CC^2 \to \CC^3, \quad \gamma_j(w,z) = (w,z, \varphi_j(w,z)) \quad \text{and} \\
\theta_j &\colon \CC^2 \to \CC^2, \quad \theta_j(w,z) = (z, \varphi_j(w,z)).
\end{align*}
The first part of \cref{register_model} then shows $\gamma_j \in \mathcal{I}^{\varrho}_{2,3,4}$, and an application of a projection onto the last two coordinates after the last layer
shows $\theta_j \in \mathcal{I}^{\varrho}_{2,2,4}$.
Let $\widehat{z_k} \in \CC^{n-1}$ denote the vector that arises from $z \in \CC^n$ by omitting the $k$th entry. 
Then, we can write 
\begin{equation*}
\eta_j(z,w,u) = \pi_k(\widehat{z_k}, \theta_j(w,z_k), u),
\end{equation*}
where $\pi_k: \CC^{n+m+1} \to \CC^{n+m+1}$ is the permutation of the entries of a vector that satisfies  
\begin{equation*}
\pi_k(\widehat{z_k}, z_k, w, u) = (z,w,u)
\end{equation*}
for every $z \in \CC^n$, $w \in \CC$ and $u \in \CC^m$.
From this representation, we clearly see 
\begin{equation*}
\eta_j \in \mathcal{I}^{\varrho}_{n+m+1,n+m+1,n+m+3},
\end{equation*}
as desired.
\end{proof}
We can now prove the final bound. 
\begin{Satz}\label{thm:main-polyharmonic}
Let $n,m \in \NN$.
	Assume that $\rr \in C(\CC;\CC)$ satisfies \cref{assumptions}.
	Then the set $\mathcal{NN}^\rr_{n,m,2n+2m+5}$ is universal.
	Moreover, if there exists a point $z_0 \in \CC$ with either
	\begin{equation}\label{eq:wirtconddd}
	\wirt \rr(z_0) \neq 0 = \wirtq \rr(z_0) \quad \text{or}\quad \wirt \rr(z_0) = 0 \neq \wirtq \rr(z_0),
	\end{equation}
	then $\mathcal{NN}^\rr_{n,m,n+m+3}$ is universal.
\end{Satz}
\begin{proof}
From the Stone--Weierstrass theorem \cite[Theorem~4.51]{folland_real_1999}, 
we know that the set of complex polynomials in $z_1,\ldots,z_n$ and $\overline{z_1},\ldots,\overline{z_n}$ is dense in $C(K;\CC)$ 
with respect to the uniform norm on any compact set $K\subset\CC^n$.
Hence, it suffices to show that each function $p:\CC^n \to \CC^m$, whose components are polynomials in 
$z_1,\ldots,z_n$ and $\overline{z_1},\ldots,\overline{z_n}$ can be uniformly approximated on $K$ by CVNNs of appropriate width.
Equivalently, it suffices to show that $p \in \overline{\mathcal{NN}^{\rr}_{n,m,W}}$, where $W$ is the desired width and the closure 
is taken with respect to the compact-open topology (see \cref{prop:equivalence,prop:comp_open_loc_unif_conv,prop:sequential}). 

Let $\mul : \CC^2 \to \CC$ be chosen according to \cref{prop: multiplication_approx} (depending on $\rr$).
From \cref{lem:thomas_model}, we infer that there exists a natural number $L \in \NN$ such that
\begin{equation*}
p = T_L \circ g_L \circ \ldots \circ T_1 \circ g_1 \circ T_0,
\end{equation*}
where $T_0 \in \Aff(\CC^n; \CC^{n+m+1})$, $T_L \in \Aff(\CC^{n+m+1}; \CC^m)$ and $T_\ell \in \Aff(\CC^{n+m+1}; \CC^{n+m+1})$ for 
every $\ell \in \setn{1,\ldots,L-1}$.
Moreover, each function $g_\ell : \CC^{n+m+1} \to \CC^{n+m+1}$ is a finite composition of the $f_k = f_k^{\mul}$ and $h_k$ as defined in \cref{lem:thomas_model}.
In view of \cref{prop:comp_loc_conv}, it hence suffices to show that for every $k \in \setn{1,\ldots,n}$, we have 
\begin{equation*}
f_k, h_k \in \overline{\mathcal{NN}^{\rr}_{n+m+1, n+m+1, 2n+2m+5}},
\end{equation*}
where the closure is taken with respect to the compact-open topology.

To this end, let $k \in \setn{1,\ldots,n}$.
Since $\rr \in C^2(\CC;\CC)$ and $\rr$ is by assumption clearly not constant, there exists a point $z_1 \in \CC$ at which $\rr$
is real differentiable with non-zero derivative.
\cref{prop:register_to_full,prop:fk} then show
\begin{equation*}
f_k \in \overline{\mathcal{I}^{\rr}_{n+m+1,n+m+1,n+m+3}} \subseteq \overline{\mathcal{NN}^{\rr}_{n+m+1, n+m+1, 2n+2m+5}}.
\end{equation*}
Moreover, since both the identity and the conjugation can be locally uniformly approximated on $\CC$ 
by shallow networks of width 2 according to \cref{prop: approx_identity_conj}, we get 
\begin{equation*}
h_k \in \overline{\mathcal{SN}^{\rr}_{n+m+1, n+m+1, 2n+2m + 2}}\subseteq \overline{\mathcal{NN}^{\rr}_{n+m+1, n+m+1, 2n+2m+5}}.
\end{equation*}
This proves that $\mathcal{NN}^{\rr}_{n+m+1, n+m+1, 2n+2m+5}$ is universal (see \cref{prop:equi}). 

Let us now assume that there exists a point $z_0 \in \CC$ with 
\begin{equation*}
\wirt \rr (z_0) \neq 0 = \wirtq \rr(z_0).
\end{equation*}
In that case, \cref{prop:register_to_full,prop:fk} yield 
\begin{equation*}
f_k \in \overline{\mathcal{I}^{\rr}_{n+m+1,n+m+1,n+m+3}} \subseteq \overline{\mathcal{NN}^{\rr}_{n+m+1, n+m+1, n+m+3}}.
\end{equation*}
Moreover, since identities can be locally uniformly approximated by shallow CVNNs of width 1 according to \cref{prop: main}, we get 
\begin{equation*}
h_k \in \overline{\mathcal{SN}^{\rr}_{n+m+1, n+m+1, n+m+2}}\subseteq \overline{\mathcal{NN}^{\rr}_{n+m+1, n+m+1, n+m+3}},
\end{equation*}
where we again used \cref{prop: approx_identity_conj} to approximate the conjugation (which might possibly require a width of 2).

It remains to deal with the case 
\begin{equation*}
\wirt \rr (z_0) = 0 \neq \wirtq \rr(z_0).
\end{equation*}
Note that $\overline{\rr}$ satisfies \cref{assumptions} and by the fundamental properties of Wirtinger derivatives (see for instance \cite[E.~1a]{kaup_holomorphic_1983}) we compute 
	\begin{equation*}
	\wirt \overline{\rr}(z_0) = \overline{\wirtq \rr (z_0)} \neq 0 \quad\text{and}\quad \wirtq \overline{\rr}(z_0) = \overline{\wirt \rr (z_0)} = 0.
	\end{equation*}
Hence, by what we have previously shown, we infer that $\mathcal{NN}^{\overline{\rr}}_{n,m,n+m+3}$ is universal. 
\cref{prop:conj_universal} then yields that $\mathcal{NN}^{\rr}_{n,m,n+m+3}$ is universal.
\end{proof}
\section{A quantitative bound in terms of the depth}\label{sec:quant}
In this section, we provide a quantitative approximation bound in terms of the depth of the considered networks.
For $n \in \NN$, we let 
\begin{equation*}
\Omega_n \defeq [-1,1]^n + \ii \cdot [-1,1]^n
\end{equation*}
 denote the $2n$-dimensional unit cube embedded into $\CC^n$. 
Let $f \in C(\Omega_n; \CC^m)$ be a given continuous function defined on that cube and let $\eps > 0$ be a prescribed approximation accuracy. 
According to the results established in \cref{chap:polyharmonic,chap:non-polyharmonic}, we can approximate the function $f$ up to arbitrary precision on $\Omega_n$ 
with deep narrow networks.
However, these statements are of a qualitative nature and do not address the question of how deep the networks have to be in order to guarantee an approximation accuracy less then $\eps$.
In this section, we prove such a quantitative statement for the case of activation functions that are smooth and non-polyharmonic on some non-empty open subset of $\CC$.
Our bound heavily relies on the \emph{modulus of continuity} of the given function $f$:
For a compact set $K \subseteq \CC^n$, a function $f \in C(K; \CC^m)$ and $h>0$, we let
\begin{equation*}
\omega(f,h) \defeq \sup \setcond{\mnorm{f(z_1) - f(z_2)}_{\CC^m}}{z_1,z_2\in K,\mnorm{z_1 - z_2}_{\CC^n} < h}.
\end{equation*}
We get the following result, which is a generalization of \cite[Proposition~48]{kratsios2022universal} to the complex-valued setting. 
\begin{Prop}\label{thm:geuchen-lemma}
Let $f \in C(\Omega_n ; \CC)$. 
Then for every $k \in \NN$ there exists $p \in \mathcal{P}_{2k}^n$ with 
\begin{equation*}
\mnorm{f - p}_{C(\Omega_n; \CC)} \leq \lr{\sqrt{2} + \frac{n}{\sqrt{2}}} \omega\lr{f, \frac{2}{\sqrt{k}}}.
\end{equation*}
Here, $\mathcal{P}_{2k}^n$ is as defined in \eqref{eq:poly}.
\end{Prop}
\begin{proof}
We define 
\begin{equation*}
\widetilde{f}: [0,1]^n + \ii \cdot [0,1]^n \to \CC, \quad \widetilde{f}(z) = f(2z - \mathds{1}),
\end{equation*}
where $\mathds{1}\in\CC^n$ is the vector with every entry equal to $1+i$.
According to \cite[Proposition~48]{kratsios2022universal}, we then observe the existence of a function $\widetilde{p_1}: \CC^n \to \RR$ with 
\begin{equation*}
\mnorm{\RE(\widetilde{f})- \widetilde{p_1}}_{C([0,1]^n + \ii \cdot [0,1]^n; \RR)} \leq \lr{1 + \frac{2n}{4}} \omega\lr{\RE(\widetilde{f}), \frac{1}{\sqrt{k}}} = 
 \lr{1 + \frac{n}{2}} \omega\lr{\RE(\widetilde{f}), \frac{1}{\sqrt{k}}}
\end{equation*}
and $\widetilde{p_1}$ is of the form 
\begin{equation*}
\widetilde{p_1}(z) = \underset{\kk \leq k}{\sum_{\kk \in \NN_0^{2n}}} a_{\kk} \prod_{j=1}^n \RE(z_j)^{\kk_j} \IM(z_j)^{\kk_{n+j}},
\end{equation*}
with $a_\kk \in \RR$ for every $\kk \in \NN_0^{2n}$ with $\kk \leq k$.
Similarly, we obtain a function $\widetilde{p_2}: \CC^n  \to \RR$ of the form 
\begin{equation*}
\widetilde{p_2}(z) = \underset{\kk \leq k}{\sum_{\kk \in \NN_0^{2n}}} b_{\kk} \prod_{j=1}^n \RE(z_j)^{\kk_j} \IM(z_j)^{\kk_{n+j}}
\end{equation*}
with $b_{\kk} \in \RR$ satisfying 
\begin{equation*}
\mnorm{\IM(\widetilde{f}) - \widetilde{p_2}}_{C([0,1]^n + \ii \cdot [0,1]^n; \RR)} \leq
 \lr{1 + \frac{n}{2}} \omega\lr{\IM(\widetilde{f}), \frac{1}{\sqrt{k}}}.
\end{equation*}
We set $\widetilde{p} \defeq \widetilde{p_1} + \ii \cdot \widetilde{p_2}$ and get 
\begin{align*}
&\mnorm{\widetilde{f} - \widetilde{p}}_{C([0,1]^n + \ii \cdot [0,1]^n ; \CC)}\\
&=\sup_{z\in [0,1]^n+\ii\cdot [0,1]^n}\sqrt{(\RE(\tilde{f}(z))-\RE(\tilde{p}(z)))^2+(\IM(\tilde{f}(z))-\IM(\tilde{p}(z)))^2}\\
&\leq \sqrt{\mnorm{\RE(\widetilde{f})- \widetilde{p_1}}_{C([0,1]^n + \ii \cdot [0,1]^n; \RR)}^2+\mnorm{\IM(\widetilde{f}) - \widetilde{p_2}}_{C([0,1]^n + \ii \cdot [0,1]^n; \RR)}^2}\\
& \leq \lr{\sqrt{2} + \frac{n}{\sqrt{2}}} \omega\lr{\widetilde{f}, \frac{1}{\sqrt{k}}}.
\end{align*}
Using the substitutions $\RE(z) = \frac{1}{2}(z + \overline{z})$ and $\IM(z) = \frac{1}{2 \ii}(z - \overline{z})$, we obtain $\widetilde{p} \in \mathcal{P}_{2k}^n$.
In the end, we set $p(z) \defeq \widetilde{p}(\frac{1}{2}z + \frac{1}{2}\mathds{1})$ and note $p \in \mathcal{P}_{2k}^n$. 
Moreover, we have 
\begin{equation*}
\mnorm{f - p}_{C(\Omega_n ; \CC)} \leq \lr{\sqrt{2} + \frac{n}{\sqrt{2}}} \omega\lr{\widetilde{f}, \frac{1}{\sqrt{k}}}.
\end{equation*}
It remains to show that $ \omega (\widetilde{f}, \frac{1}{\sqrt{k}})=\omega(f, \frac{2}{\sqrt{k}})$. 
To this end, we observe 
\begin{align*}
\omega \lr{\widetilde{f}, \frac{1}{\sqrt{k}}} &= 
\sup \setcond{\abs{\widetilde{f}(\widetilde{z_1}) - \widetilde{f}(\widetilde{z_2})} }{ \widetilde{z_1}, \widetilde{z_2} \in [0,1]^n + \ii \cdot [0,1]^n, \ \mnorm{\widetilde{z_1} - \widetilde{z_2}}_{\CC^n} \leq \frac{1}{\sqrt{k}}} \\
&= \sup \setcond{\abs{f(2\widetilde{z_1}-1) - f(2\widetilde{z_2}-1)} }{ \widetilde{z_1}, \widetilde{z_2} \in [0,1]^n + \ii \cdot [0,1]^n, \ \mnorm{\widetilde{z_1} - \widetilde{z_2}}_{\CC^n} \leq \frac{1}{\sqrt{k}}} \\
&= \sup \setcond{\abs{f(z_1) - f(z_1)} }{ z_1, z_2 \in \Omega_n, \ \mnorm{\frac{1}{2}z_1 + \frac{1}{2} - \lr{\frac{1}{2} z_2 + \frac{1}{2}}}_{\CC^n} \leq \frac{1}{\sqrt{k}}} \\
&= \omega\lr{f, \frac{2}{\sqrt{k}}}.\qedhere
\end{align*}
\end{proof}
The following result shows how polynomials can be approximated by shallow CVNNs. 
Remarkably, the size of the networks needed does only depend on the degree of the polynomial and not on the approximation accuracy. 
Moreover, if one aims to approximate polynomials from a bounded subset of $\mathcal{P}_k^n$ (with respect to any norm on $\mathcal{P}_k^n$), one 
can choose the weights connecting the input and the hidden layer of the shallow network independent of the particular polynomial $p$ which is to be approximated;
only the weights connecting hidden and output layer have to be adjusted to $p$.
\begin{Prop}[{cf. \cite[Theorem~3.1]{Geuchen2023}}]\label{prop:neurips}
Let $k,n \in \NN$, $\eps > 0$ and $\rr: \CC \to \CC$ be smooth and non-polyharmonic on a non-empty open
set $U \subseteq \CC$.
Let $\mathcal{P}' \subseteq \mathcal{P}_k^n$ be bounded (with respect to some norm on $\mathcal{P}^n_k$) and set $N \defeq (4k + 1)^{2n}$. 
Then there exists a map $\varphi \in \Aff(\CC^n ; \CC^N)$ with the following property:
For every $p \in \mathcal{P}'$ there exists a map $\psi \in \Aff(\CC^N; \CC)$ with 
\begin{equation*}
\mnorm{ p - \lr{\psi \circ \rr^{\times N} \circ \varphi}}_{C(\Omega_n; \CC)} \leq \eps. 
\end{equation*}
\end{Prop}
We can now prove the desired quantitative approximation statement. 
It is based on \cref{thm:geuchen-lemma,prop:neurips} and uses the fact that each shallow network can be approximated 
up to arbitrary precision by deep narrow CVNNs according to 
\cref{prop:register_to_full,register_model}.
The depth of this CVNN is determined by the number of hidden neurons in the shallow network, which we can quantify according to \cref{thm:geuchen-lemma,prop:neurips}.
To formulate the final result, we further introduce the notation 
\begin{equation}\label{eq:inverse_mod}
\omega^{-1}(f, \eps) \defeq \sup\setcond{\delta > 0 }{ \omega(f,\delta) \leq \eps}
\end{equation}
for $\eps > 0$ and $f \in C(\Omega_n; \CC^m)$.
\begin{Satz}\label{thm:quant}
Let $\rr \in C(\CC; \CC)$ be smooth and non-polyharmonic on some non-empty open set $\emptyset \neq U \subseteq \CC$.
Moreover, let $f =(f_1,\ldots,f_m)\in C(\Omega_n; \CC^m)$ and $\eps > 0$ be given. 
Then there exists a number $N \in \NN$ with 
\begin{equation}\label{eq:depth-bound}
N \leq \lr{ 32\cdot \left[\omega^{-1}\lr{f, \frac{\eps}{3 \cdot \sqrt{2m} \cdot \lr{1 + \frac{n}{2}}}}\right]^{-2} + 9}^{2n}
\end{equation}
and a network $\Phi \in \mathcal{NN}^\rr_{n,m,2n+2m+1,N}$ with $\mnorm{f - \Phi}_{C(\Omega_n; \CC^m)} \leq \eps$.

Moreover, if there exists a point $z_0 \in \CC$ where $\rr$ is real differentiable with 
\begin{equation*}
\wirt \rr (z_0) \neq 0 = \wirtq \rr(z_0),
\end{equation*}
we may pick $\Phi \in \mathcal{NN}^\rr_{n,m,n+m+1,N}$.
If, on the other hand, we have 
\begin{equation*}
\wirt \rr (z_0) = 0 \neq \wirtq \rr(z_0),
\end{equation*}
for some $z_0\in\CC^n$, we may pick $\Phi \in \mathcal{NN}^\rr_{n,m,n+m+1,2N}$.
\end{Satz}
\begin{proof}
We set 
\begin{equation*}
k \defeq \left\lceil \left(\omega^{-1}\left(f, \frac{\eps}{3 \cdot \sqrt{2m}\left(1 + \frac{n}{2}\right)}\right)\right)^{-2} \cdot 4\right\rceil.
\end{equation*}
Note that we have 
\begin{equation*}
k \geq \left(\omega^{-1}\left(f, \frac{\eps}{3 \cdot \sqrt{2m}\left(1 + \frac{n}{2}\right)}\right)\right)^{-2} \cdot 4,
\end{equation*}
which implies 
\begin{equation*}
\frac{2}{\sqrt{k}} \leq \omega^{-1}\left(f, \frac{\eps}{3 \cdot \sqrt{2m}\left(1 + \frac{n}{2}\right)}\right).
\end{equation*}
By \eqref{eq:inverse_mod}, we get 
\begin{equation*}
\omega \left(f, \frac{2}{\sqrt{k}}\right) \leq \frac{\eps}{3 \cdot \sqrt{2m}\cdot \left(1 + \frac{n}{2}\right)}.
\end{equation*}
According to \cref{thm:geuchen-lemma}, there exist polynomials $p_1,\ldots, p_m \in \mathcal{P}^n_{2k}$ with 
\begin{equation*}
\mnorm{f_j - p_j}_{C(\Omega_n;\CC)} \leq \lr{\sqrt{2} + \frac{n}{\sqrt{2}}} \omega\lr{f_j, \frac{2}{\sqrt{k}}}
\end{equation*}
for every $j \in \setn{1,\ldots,m}$.
Letting $p \defeq (p_1,\ldots, p_m)$, we then have 
\begin{align*}
\mnorm{f-p}_{C(\Omega_n; \CC^m)}&\leq \lr{\sum_{j=1}^m \mnorm{f_j-p_j}_{C(\Omega_n;\CC)}^2}^{1/2}\\
&\leq  \lr{\sqrt{2} + \frac{n}{\sqrt{2}}} \cdot \lr{\sum_{j=1}^m \left[ \omega\lr{f_j, \frac{2}{\sqrt{k}}}\right]^2}^{1/2} \\
&\leq \lr{\sqrt{2} + \frac{n}{\sqrt{2}}} \cdot \lr{\sum_{j=1}^m \left[ \omega\lr{f, \frac{2}{\sqrt{k}}}\right]^2}^{1/2} \\
&= \lr{\sqrt{2} + \frac{n}{\sqrt{2}}} \cdot \sqrt{m} \cdot \omega\lr{f, \frac{2}{\sqrt{k}}}\leq \frac{\eps}{3}.
\end{align*}
We set $N \defeq (8k + 1)^{2n}$.
Applying \cref{prop:neurips} to the finite (and therefore bounded) set given by 
$\mathcal{P}' \defeq \setn{p_1,\ldots, p_m}$ yields the existence of functions $\varphi \in \Aff(\CC^n ; \CC^N)$
and $\psi \in \Aff(\CC^N; \CC^m)$ with 
\begin{equation*}
\mnorm{p - \lr{\psi \circ \rr^{\times N} \circ \varphi}}_{C(\Omega_n; \CC^m)} \leq \frac{\eps}{3}.
\end{equation*}
Note that $N$ is independent of $m$ since the weights connecting the input and the hidden layer can be chosen independent of the polynomial $p_j$ 
(see \cref{prop:neurips}).
Since $\rr$ is smooth and non-polyharmonic on a non-empty open set (in particular not constant) there exists $z_0 \in \CC$ such that $\rr$
is real differentiable at $z_0$ with non-zero derivative. 
According to \cref{register_model} and \cref{prop:register_to_full}\ref{item:3} we have 
\begin{equation}\label{eq:split}
\psi \circ \rr^{\times N} \circ \varphi \in \mathcal{I}^{\rr}_{n,m,n+m+1, N} \subseteq \overline{\mathcal{NN}^{\rr}_{n,m,2n+2m+1,N}}.
\end{equation}
Hence, there exists $\Phi \in \mathcal{NN}^\rr_{n,m,2n+2m+1,N}$ with 
\begin{equation*}
\mnorm{\Phi - \lr{\psi \circ \rr^{\times N} \circ \varphi}}_{C(\Omega_n; \CC^m)} \leq \frac{\eps}{3}.
\end{equation*}
By the triangle inequality, we get 
\begin{align*}
\mnorm{f - \Phi}_{C(\Omega_n; \CC^m)}&\leq \mnorm{f - p}_{C(\Omega_n; \CC^m)}+\mnorm{p - \lr{\psi \circ \rr^{\times N} \circ \varphi}}_{C(\Omega_n; \CC^m)} +\mnorm{\Phi - \lr{\psi \circ \rr^{\times N} \circ \varphi}}_{C(\Omega_n; \CC^m)} \\
&\leq \frac{\eps}{3}+\frac{\eps}{3}+\frac{\eps}{3}\leq\eps,
\end{align*}
as desired. 

It remains to estimate the depth $N$ of $\Phi$.
Note that we have 
\begin{equation*}
k \leq \left(\omega^{-1}\left(f, \frac{\eps}{3 \cdot \sqrt{2m}\left(1 + \frac{n}{2}\right)}\right)\right)^{-2} \cdot 4 +1
\end{equation*}
by definition of $k$.
Combining this with $N = (8k+1)^{2n}$, we obtain the upper bound \eqref{eq:depth-bound} for $N$.

The case that there exists $z_0 \in \CC$ such that $\rr$ is differentiable at $z_0$ with 
\begin{equation*}
(\wirt \rr(z_0), \wirtq \rr(z_0)) \neq (0,0)
\end{equation*}
follows analogously, by using \cref{prop:register_to_full}\ref{item:1} and \ref{item:2} in \eqref{eq:split}.
\end{proof}

\section{Necessity of our assumptions}\label{chap:optimality}

The proof of \cref{thm:intro} is not yet complete.
So far, we have proven that activation functions $\rr \in C(\CC;\CC)$ which are neither holomorphic, nor antiholomorphic, nor $\RR$-affine yield universality of CVNNs of width $2n+2m+5$, with indicated improvements under additional assumptions.
The necessity part is done in this section: If $\rr$ is holomorphic, antiholomorphic, or $\RR$-affine, then even the set of CVNNs using activation function $\rr$ with \emph{arbitrary widths and depths} is not universal, cf.~\cref{thm:holo_anti_aff}.
Furthermore, \cref{thm:intro} states that under the mentioned constraints on the activation function, a width of $2n+2m+5$ is sufficient for universality of CVNNs with input dimension $n$ and output dimension $m$.
But could we have done better?
In \cref{thm:lower_bound} below, we show that for a family of real-valued activation functions, the set $\mathcal{NN}^\rr_{n,m,W}$ is not universal when $W<\max\setn{2n,2m}$.
In our final result \cref{thm:necessary_diff}, we show that real differentiability of the activation function at one point with non-vanishing derivative is not necessary for the universal approximation property of deep narrow CVNNs.

We prepare the proof of \cref{thm:holo_anti_aff} by two lemmas, the first of which is about uniform convergence of $\RR$-affine functions.
\begin{Lem}
\label{thm:R-affine}
Let $n,m \in \NN$, $(f_k)_{k \in \NN}$ be a sequence of $\RR$-affine functions from $\RR^n$ to $\RR^m$ and $f:\RR^n \to \RR^m$.
Let $(f_k)_{k \in \NN}$ converge locally uniformly to $f$.
Then $f$ is $\RR$-affine too.
\end{Lem}
\begin{proof}
Let $A_k \in \RR^{m \times n}$ and $b_k \in \RR^m$ such that $f_k(x) = A_k x + b_k$.
Let $b \defeq f(0)$.
Then we have $b_k = f_k(0) \to f(0) = b$.

Furthermore we see for every $j \in \setn{1,\ldots,n}$ that
\begin{equation*}
\mnorm{ A_k e_j - A_\ell e_j}_{\RR^m} \leq \mnorm{ A_k e_j + b_k - A_\ell e_j - b_\ell}_{\RR^m} + \mnorm{b_k - b_\ell}_{\RR^m} \to 0
\end{equation*}
as $k,\ell \to \infty$, uniformly over $j$, meaning $\max_{j \in\setn{1,\ldots,n}} \mnorm{A_ke_j - A_\ell e_j}_{\RR^m} \to 0$ as $k,\ell \to \infty$.
Here $e_j$ denotes the element of $\RR^m$ whose entries are $0$ except for the $j$th which is $1$.

Consequently, $(A_k)_{k \in \NN}$ is a Cauchy sequence and thus converges to some $A \in \RR^{m \times n}$.
We claim $f(x) = Ax + b$ for every $x \in \RR^n$.
Indeed, this follows from
\begin{align*}
\mnorm{A_k x + b_k - A x + b}_{\RR^m} \leq \mnorm{A_k x - A x}_{\RR^m} + \mnorm{b_k - b}_{\RR^m}\leq \mnorm{A_k - A}_{\RR^{m\times n}}  \mnorm{x}_{\RR^m} + \mnorm{b_k - b}_{\RR^m}  \to 0.
\end{align*}
as $k\to\infty$.
\end{proof}
Our second lemma in preparation of the proof of \cref{thm:holo_anti_aff} concerns locally uniform limits of sequences of functions that are either holomorphic or antiholomorphic.
\begin{Lem}
\label{thm:hol_antihol}
Let $\mathcal{F} \defeq \setcond{F: \CC \to \CC }{F \text{ holomorphic or antiholomorphic}}$ and $(f_k)_{k \in \NN}$ be a sequence of functions with $f_k \in \mathcal{F}$ for every $k \in \NN$.
Let $f: \CC \to \CC$ be such that $f_k \to f$ locally uniformly.
Then it holds $f \in \mathcal{F}$.
\end{Lem}
\begin{proof}
We distinguish two cases:
\begin{enumerate}[label={(\roman*)},leftmargin=*,align=left,noitemsep]
\item{If there exists a subsequence of $(f_k)_{k \in \NN}$ consisting of holomorphic functions, the limit $f$ of this subsequence also has to be holomorphic (see for instance \cite[Theorem~10.28]{Rudin1987}).}
\item{If there exists a subsequence of $(f_k)_{k \in \NN}$ consisting of antiholomorphic functions, the limit $f$ of this subsequence also has to be antiholomorphic, where we again apply \cite[Theorem~10.28]{Rudin1987} to the complex conjugates of the functions in this subsequence.}
\end{enumerate}
\end{proof}
The necessity part of \cref{thm:intro} is covered by the following theorem.
\begin{Satz}\label{thm:holo_anti_aff}
Let $n,m \in \NN$ and $\rr: \CC \to \CC$ be holomorphic or antiholomorphic or $\RR$-affine.
Then $\mathcal{NN}^\rr_{n,m}\defeq \bigcup\limits_{W \in \NN} \mathcal{NN}^\rr_{n,m,W}$ is not universal.
\end{Satz}
\begin{proof}
It suffices to show the claim for $m=1$, since local uniform approximation in $\CC^m$ means componentwise local uniform approximation.

We start with the case $n=1$.
If $\rr$ is holomorphic or antiholomorphic it follows that all the elements of $\mathcal{NN}^\rr_{1,1}$ are holomorphic or antiholomorphic (see, e.g., \cite[Proof of Eq.~(4.15), p.~28]{Voigtlaender2022}.
But then it follows from \cref{thm:hol_antihol,prop:equivalence} that $\mathcal{NN}^\rr_{1,1}$ is not universal.
If $\rr$ is $\RR$-affine, each element of $\mathcal{NN}^\rr_{1,1}$ is $\RR$-affine (as the composition of $\RR$-affine functions).
By \cref{thm:R-affine,prop:equivalence} it follows that $\mathcal{NN}^\rr_{1,1}$ is not universal.

The case $n > 1$ can be reduced to the case $n=1$ in the following way: Assume that $\mathcal{NN}^\rr_{n,1}$ is universal and pick any arbitrary function $f \in C(\CC;\CC)$.
Let $\pi: \CC^n \to \CC$, $\pi(z_1,\ldots,z_n) = z_1$ and $\widetilde{\pi}:\CC \to \CC^n$, $\widetilde{\pi}(z) = (z,0,\ldots,0)$.
Note that it holds $\pi \circ \widetilde{\pi} = \id_\CC$.
By assumption, there exists a sequence $(g_k)_{k \in \NN}$ with $g_k \in \mathcal{NN}^\rr_{n,1}$  for $k \in \NN$ and $g_k \to f \circ \pi$ locally uniformly.
From \cref{prop:comp_loc_conv} it follows $g_k \circ \widetilde{\pi} \to f$ locally uniformly.
Since $g_k \circ \widetilde{\pi} \in \mathcal{NN}^\rr_{1,1}$ for every $k \in \NN$ it follows that $\mathcal{NN}^\rr_{1,1}$ is universal, in contradiction to what has just been shown.
\end{proof}

In the previous sections, we showed that for a large class of activation functions a width of $2n+2m+5$ is sufficient for universality of CVNNs with input dimension $n$ and output dimension $m$.
Following the lines of \cite[Lemma~1]{Cai2022}, we show next that for some activation functions, a width of at least $\max\setn{2n,2m}$ is necessary to guarantee universality.

\begin{Satz}\label{thm:lower_bound}
Let $n,m\in\NN$.
\begin{enumerate}[label={(\roman*)},leftmargin=*,align=left,noitemsep]
\item{Let $\phi \in C(\RR;\CC)$, and $\rr:\CC\to\CC$ be given by $\rr(z) \defeq \phi(\RE(z))$.
Then $\mathcal{NN}^\rr_{n,m,2n-1}$ is not universal.\label{input}}
\item{Let $\rr:\CC\to\RR$.
Then $\mathcal{NN}^\rr_{n,m,2m-1}$ is not universal.\label{output}}
\end{enumerate}
\end{Satz}

\begin{proof}
We start with \ref{input}.
Let $K \defeq [-2,2]^n+\ii [-2,2]^n \subseteq \CC^n$ and $f(z) \defeq (\mnorm{z}_{\CC^n},0,\ldots,0)$ for $z \in \CC^n$.
Let $g \in \mathcal{NN}^\rr_{n,m,2n-1}$ be arbitrary.
From the definition of $\rr$, it follows that we may write $g$ as
\begin{equation*}
g(z) = \psi (\RE(V z) + b),
\end{equation*}
where $\psi:\CC^{2n-1} \to \CC^m$ is some function, $V \in \CC^{(2n-1) \times n}$, $b \in \RR^{2n-1}$, and the real part $\RE$ is taken componentwise.
Interpreting $\RE \circ V$ as an $\RR$-linear function from $\RR^{2n}$ to $\RR^{2n-1}$, we conclude from $2n-1 < 2n$ that there exists $v \in \CC^n$ with $\mnorm{v}_{\CC^n} = 1$ satisfying $\RE (Vv) = 0$ and hence
\begin{equation}
\label{eq:kernel_vector}
g(z+v) = g(z) \text{ for any } z \in \CC^n.
\end{equation}
We set 
\begin{equation*}
A \defeq \setcond{z \in \CC^n }{ \mnorm{z}_{\CC^n} < \frac{1}{10}} \quad \text{and} 
\quad B \defeq \setcond{ z \in \CC^n }{\mnorm{z-v}_{\CC^n} < \frac{1}{10}}
\end{equation*}
and compute
\begin{align*}
\int_K \mnorm{f(z) - g(z)}_{\CC^m} \dd z &\geq \int_A \mnorm{ f(z) - g(z) }_{\CC^m} \dd z + \int_B  \mnorm{f(z) - g(z)}_{\CC^m}\dd z \\
\overset{B = A + v}&{=} \int_A \lr{\mnorm{ f(z) - g(z)}_{\CC^m} + \mnorm{ f(z+ v) - g(z+v)}_{\CC^m} }\dd z \\
\overset{\eqref{eq:kernel_vector}}&{\geq} \int_A \mnorm{ f(z) - f(z+v)}_{\CC^m} \dd z \geq 0.8 \cdot \lambda^{2n} (A)
\end{align*}
with $\lambda^{2n}$ denoting the $2n$-dimensional Lebesgue measure.
In the last inequality we used
\begin{equation*}
\abs{ \mnorm{z}_{\CC^n}  - \mnorm{z + v}_{\CC^n}} \geq \mnorm{z + v}_{\CC^n} - \mnorm{z}_{\CC^n}   \geq \mnorm{v}_{\CC^n} - 2 \mnorm{z}_{\CC^n} \geq 0.8.
\end{equation*}
Hence it follows that $\mathcal{NN}^\rr_{n,m,2n-1}$ is not dense in $C(K;\CC^m)$ with respect to the $L^1$-norm and thus, using Hölder's inequality, it follows that $\mathcal{NN}^\rr_{n,m,2n-1}$ is not dense in $C(K;\CC^m)$ with respect to the $L^p$-norm for any $p \in [1, \infty]$, so in particular for $p = \infty$ which shows that $\mathcal{NN}^\rr_{n,m,2n-1}$ is not universal.

Next, we prove \ref{output}.
To this end, we construct a function $f \in C(\CC^n;\CC^m)$, a compact set $K\subset\CC^n$, and a number $\eps>0$ such that
\begin{equation*}
\sup_{z\in K}\mnorm{f(z)-g(z)}_{\CC^m}\geq\eps
\end{equation*}
for all $g\in \mathcal{NN}^\rr_{n,m,2m-1}$.
For a moment, fix $g\in \mathcal{NN}^\rr_{n,m,2m-1}$.
Since the activation function $\rr$ is real-valued, the output of the last but one layer of $g$ is a function $\psi:\CC^n\to\RR^{2m-1}$.
Also, there exist a matrix $V\in\CC^{m\times (2m-1)}$ and a vector $b\in\CC^m$ such that $g(z)=V\psi(z)+b$.
With $\CC^m\cong \RR^{2m}$, we may view the restriction of $V$ to $\RR^{2m-1}$ as an $\RR$-linear map $\RR^{2m-1}\to\RR^{2m}$, 
and the range $\setcond{g(z)}{z\in\CC}$ of $g$ is thus contained in a $(2m-1)$-dimensional affine subspace $U=U(g)$ of $\RR^{2m}$.
As
\begin{equation*}
\sup_{z\in K}\mnorm{f(z)-g(z)}_{\RR^{2m}}\geq \sup_{z\in K}\inf_{u\in U(g)}\mnorm{f(z)-u}_{\RR^{2m}},
\end{equation*}
it is sufficient for our purposes to find a function $f \in C(\CC^n;\CC^m)$, a compact set $K\subset\CC^n$, and a number $\eps>0$ such that
\begin{equation*}
\inf_{U}\sup_{z\in K}\inf_{u\in U}\mnorm{f(z)-u}_{\RR^{2m}}\geq \eps
\end{equation*}
where the outermost infimum traverses the $(2m-1)$-dimensional affine subspaces $U$ of $\RR^{2m}$.
This is achieved by a function $f$ whose range $\setcond{f(z)}{z\in\CC^n}$ is not contained in any $(2m-1)$-dimensional affine subspace $U$ of $\RR^{2m}$.
A semi-explicit construction is like this:

Let $K\defeq\setcond{(\lambda,0,\ldots,0)}{\lambda\in\RR,0\leq\lambda\leq 1}\subseteq\CC^n$ and
\begin{equation*}
f_1:\CC^n\to [0,1],\quad f_1(z_1,\ldots,z_n)=\max\setns{0,\min\setns{1,\RE(z_1)}}.
\end{equation*}
Further let $f_2:[0,1]\to\CC^m$ be a parameterization of a curve that along the edges of the cube $Q\defeq [0,1]^m+\ii[0,1]^m\subseteq \CC^m\cong \RR^{2m}$ passes through all of its vertices $\setn{0,1}^m+\ii \setn{0,1}^m\subseteq \CC^m\cong \RR^{2m}$, and $f=f_2\circ f_1$.
From \cite[Table~1]{Brandenberg2005}, we deduce
\begin{equation*}
\inf_{U}\sup_{z\in K}\inf_{u\in U}\mnorm{f(z)-u}_{\RR^{2m}}\geq \frac{1}{2}
\end{equation*}
and this finishes the proof.
\end{proof}
Note that there are non-polyharmonic functions like $\rr(z)=\ee^{\RE(z)}$ but also polyharmonic functions that are neither holomorphic, nor anti-holomorphic, nor $\RR$-affine like $\rr(z)=\RE(z)^2$ that meet the assumptions made in \cref{thm:lower_bound}.

Following the ideas from \cite[Proposition~4.15]{KidgerLy2020} we also want to add a short note on the necessity of the differentiability of the activation function.
It turns out that the differentiability of the activation function is \emph{not} a necessary condition for the fact that narrow networks with width $n+m+1$ using this activation function have the universal approximation property.
The proof is in fact identical to the proof presented in \cite{KidgerLy2020}.
However, we include a detailed proof to clarify that the reasoning also works in the case of activation functions $\CC \to \CC$.
\begin{Satz}\label{thm:necessary_diff}
Take any function $w \in C(\CC; \CC)$ which is bounded and nowhere real differentiable.
Then $\rr(z) \defeq \sin(z) + w(z)  \exp(-z)$ is also nowhere differentiable and $\mathcal{NN}^\rr_{n,m,n+m+1}$ is universal.
\end{Satz}
\begin{proof}
Since $\rr$ is non-polyharmonic (since it is nowhere differentiable) it suffices to show that the identity function can be uniformly approximated on compact sets using compositions 
that have the form $\psi \circ \rr \circ \phi$ with $\phi, \psi \in \Aff(\CC; \CC)$.
Then the statement can be derived similarly to the proof of \cref{main_theorem_non_poly_classical_reg}\ref{main_theorem_non_poly_classical_reg_item1}.

Therefore, take any compact set $K \subseteq \CC$ and $\eps > 0$.
Choose $M_1 > 0$ with $\abs{z} \leq M_1$ for every $z \in K$.
Take $h> 0$ arbitrary and consider
\begin{equation*}
\sup_{z \in K \setminus \setn{0}}\abs{\frac{\sin(hz) - hz}{h}} \leq \sup_{z \in K \setminus \setn{0}}M_1 \abs{\frac{\sin(hz)}{hz} - 1} \to 0
\end{equation*}
as $h\to 0$.
Therefore, we may take $h > 0$ with
\begin{equation*}
\abs{\frac{\sin(hz) - hz}{h}} < \frac{\eps}{2}
\end{equation*}
for every $z \in K$.
Furthermore, choose $M_2 > 0$ with $\abs{w(z)} \leq M_2$ for every $z \in \CC$ and pick $k \in \NN$ large enough such that
\begin{equation*}
\frac{\abs{\exp(-hz)}\abs{\exp(-2\piup k)}}{h} < \frac{\eps}{2M_2}
\end{equation*}
for all $z \in K$.
Hence we derive
\begin{align*}
&\norel \abs{\frac{\sin(hz + 2\piup k) + w(hz + 2\piup k)\exp(-hz - 2\piup k)}{h} - z } \\
&= \abs{\frac{\sin(hz ) + w(hz + 2\piup k)\exp(-hz - 2\piup k)}{h} - z }  \\
&\leq \abs{\frac{\sin(hz) - hz}{h}} + M_2\frac{\abs{\exp(-hz)} \abs{ \exp(-2\piup k)}}{h} < \eps.
\end{align*}
Thus, we get the claim by defining $\phi(z) \defeq hz + 2\piup k$ and $\psi(z) \defeq \frac{1}{h}z$.
\end{proof}

\appendix

\section{Topological notes on locally uniform convergence}
\label{loc_uniform_convergence}

In this appendix, we discuss the relationship between locally uniform convergence, the compact open topology, and the universal approximation property introduced in \Cref{def:universal_approximation_property}.
Note that \cite[Appendix~B]{Park2022} is another account on the same topic.

Although locally uniform convergence can be studied more generally for functions defined on a topological space and taking values in a metric space, we restrict ourselves to functions $\CC^n\to\CC^m$.
\begin{Def}
	Let $(f_k)_{k \in \NN}$ be a sequence of functions $f_k : \CC^n \to \CC^m$ and $f: \CC^n \to \CC^m$.
	The sequence $(f_k)_{k \in \NN}$ converges locally uniformly to $f$, if for every compact set $K \subseteq \CC^n$ we have
\begin{equation*}
\sup_{z\in K} \mnorm{f_k(z) - f (z)}_{\CC^m} \xrightarrow{k \to \infty} 0.
\end{equation*}
\end{Def}

There is a certain equivalence between locally uniform convergence and the universal approximation property introduced in \Cref{def:universal_approximation_property}.
\begin{Prop}
\label{prop:equivalence}
Let $\mathcal{F} \subseteq C(\CC^n;\CC^m)$ and $f \in C(\CC^n;\CC^m)$.
Then the following are equivalent:
\begin{enumerate}[label={(\roman*)},leftmargin=*,align=left,noitemsep]
\item{\label{approximation}For every compact set $K \subseteq \CC^n$ and $\eps > 0$, there exists a function $g \in \mathcal{F}$ satisfying
\begin{equation*}
\mnorm{f - g}_{C(K;\CC^m)}< \eps.
\end{equation*}}
\item{\label{convergence}There exists a sequence $(f_k)_{k \in \NN}$ with $f_k \in \mathcal{F}$ for $k \in \NN$ such that $(f_k)_{k \in \NN}$ converges locally uniformly to $f$.}
\end{enumerate}
\end{Prop}
\begin{proof}
We start with the implication \ref{approximation}$\Rightarrow$\ref{convergence}.
Let $f \in C(\CC^n;\CC^m)$.
For every $k \in \NN$, choose $f_k \in \mathcal{F}$ with
\begin{equation*}
\mnorm{f_k - f}_{C(\bc{0}{k};\CC^m)} \leq \frac{1}{k}.
\end{equation*}
Then $(f_k)_{k \in \NN}$ converges locally uniformly to $f$, since every compact set $K \subseteq \CC^n$ is contained in the closed ball
 $\bc{0}{k}$ with radius $k$ and center $0$ for all $k \geq J$ for some $J \in \NN$.

Now we show the implication \ref{convergence}$\Rightarrow$\ref{approximation}.
For any compact set $K \subseteq \CC^n$ and $\eps > 0$ we know by definition of locally uniform convergence that there exists $k \in \NN$ satisfying
\begin{equation*}
\mnorm{f_k - f}_{C(K;\CC^m)} < \eps.
\end{equation*}
Since $f_k \in \mathcal{F}$, this shows \ref{approximation}.
\end{proof}

Next, we show that locally uniform convergence of sequences $(f_k)_{k \in \NN}$ of elements $f_k\in C(\CC^n;\CC^m)$ coincides with convergence with respect to the \emph{compact-open topology}, cf.~\cite[Definition~XII.1.1]{dugundji_topology_1987}.
Hence, the compact-open topology is the topology in charge when we speak about universality of a set of continuous functions.
\begin{Def} \label{def:compact_open}
For each pair of sets $A \subseteq \CC^n, B \subseteq \CC^m$, we denote
\begin{equation*}
(A,B) \defeq \setcond{f \in C(\CC^n;\CC^m)}{ f(A) \subseteq B}.
\end{equation*}
The \emph{compact-open topology} on $C(\CC^n;\CC^m)$ is then the smallest topology containing the sets $(K,V)$, where $K \subseteq \CC^n$ is compact and $V \subseteq \CC^m$ is open.
\end{Def}

That this topology indeed induces locally uniform convergence is a direct consequence of \cite[Theorem~XII.7.2]{dugundji_topology_1987} and $\CC^m$ being a metric space.
\begin{Prop}
\label{prop:comp_open_loc_unif_conv}
Let $(f_k)_{k \in \NN}$ be a sequence of functions with $f_k \in C(\CC^n;\CC^m)$ and $f \in C(\CC^n;\CC^m)$.
Then the following statements are equivalent.
\begin{enumerate}[label={(\roman*)},leftmargin=*,align=left,noitemsep]
\item{The sequence $(f_k)_{k\in\NN}$ converges to $f$ in the compact-open topology.}
\item{The sequence $(f_k)_{k\in \NN}$ converges to $f$ locally uniformly.}
\end{enumerate}
\end{Prop}

The next result says that in $C(\CC^n; \CC^m)$ with the compact-open topology,
closures of subsets can be characterized by limits of sequences.
\begin{Prop}\label{prop:sequential}
Let $f \in C(\CC^n; \CC^m)$ and $\mathcal{F} \subseteq C(\CC^n; \CC^m)$.
Denote by $\overline{\mathcal{F}}$ the closure of $\mathcal{F}$ with respect to the compact-open topology.  
Then $f \in \overline{\mathcal{F}}$ if and only if there exists a sequence $(f_k)_{k \in \NN}$ with $f_k \in \mathcal{F}$ for $k \in \NN$ that converges to $f$ in the compact-open topology.
\end{Prop} 
\begin{proof}
Since both $\CC^n$ and $\CC^m$ are second countable, we infer by \cite[Theorem~XII.5.2]{dugundji_topology_1987} that $C(\CC^n; \CC^m)$ equipped with the compact-open topology is second countable, too.
For second countable topological spaces it is well-known that closures of subsets can be characterized by limits of sequences;
see \cite[Ex.~3~on~p.~186 and~Theorem~X.6.2]{dugundji_topology_1987}.
\end{proof}

In particular, \Cref{prop:equivalence,prop:sequential,prop:comp_open_loc_unif_conv} yield the following equivalence, which we state in the following 
proposition.
\begin{Prop}\label{prop:equi}
Let $\mathcal{F} \subseteq C(\CC^n; \CC^m)$. 
Then the following are equivalent:
\begin{enumerate}[label={(\roman*)},leftmargin=*,align=left,noitemsep]
\item{The set $\mathcal{F}$ has the universal approximation property.}
\item{For every $f \in C(\CC^n;\CC^m)$, there exists a sequence $(f_k)_{k \in \NN}$ of elements $f_k \in \mathcal{F}$ such that $(f_k)_{k \in \NN}$ converges locally uniformly to $f$.}
\item{The set $\mathcal{F}$ is dense in $C(\CC^n;\CC^m)$ with respect to the compact-open topology.}
\end{enumerate}
\end{Prop}

In the present paper, it is of particular importance that the composition of functions is compatible with locally uniform convergence.
\begin{Prop}
\label{prop:comp_loc_conv}
Let $(f_k)_{k \in \NN}$ and $(g_k)_{k \in \NN}$ be two sequences of functions with $f_k \in C(\CC^{n_1};\CC^{n_2})$ and $g_k \in C(\CC^{n_2};\CC^{n_3})$ for $k \in \NN$.
Let $f \in C(\CC^{n_1};\CC^{n_2})$ and $g \in C(\CC^{n_2};\CC^{n_3})$ such that $f_k \to f$ and $g_k \to g$ locally uniformly.
Then we have
\begin{equation*}
g_k \circ f_k \xrightarrow{k\to \infty} g \circ f
\end{equation*}
locally uniformly.
\end{Prop}
\begin{proof}
Using \cite[Theorem~XII.2.2]{dugundji_topology_1987}, we know that the map
\begin{equation*}
C(\CC^{n_1};\CC^{n_2}) \times C(\CC^{n_2};\CC^{n_3}) \to C(\CC^{n_1}; \CC^{n_3}), \quad  (h_1,h_2) \mapsto h_2 \circ h_1
\end{equation*}
is continuous, where each space $C(\CC^{n_j};\CC^{n_k})$ is equipped with the compact-open topology and Cartesian products of spaces are equipped with the product topology.
Note here that we use the fact that $\CC^{n_1}$ and $\CC^{n_3}$ are Hausdorff spaces and $\CC^{n_2}$ is locally compact.
Then the claim follows from \Cref{prop:comp_open_loc_unif_conv}.
\end{proof}
Note that the statement of \Cref{prop:comp_loc_conv} can inductively be extended to the composition of $L$ functions, where $L$ is any natural number.

\section{Taylor expansion using Wirtinger derivatives}

In this appendix we give some details about the Taylor expansion introduced in \Cref{thm:taylor}.
Furthermore we show that an activation function which is not $\RR$-affine necessarily admits a point where one of the second-order Wirtinger derivatives does not vanish.
We begin by restating and proving \Cref{thm:taylor}.
\begin{Lem}
\label{app:taylor}
Let $\rr\in C(\CC;\CC)$ and $z,z_0\in\CC$.
If $\rr$ is real differentiable at $z_0$, then
\begin{equation}
\label{app:taylor-1}
\rr(z+z_0)=\rr(z_0)+\wirt\rr(z_0)z+\wirtq\rr(z_0)\overline{z}+\Theta_1(z)
\end{equation}
for a function $\Theta_1:\CC\to\CC$ with $\lim_{\CC\setminus\setn{0}\ni z\to 0}\frac{\Theta_1(z)}{z}=0$.
If  $\rr\in C^2(\CC;\CC)$, then
\begin{equation}
\label{app:taylor-2}
\rr(z+z_0)=\rr(z_0)+\wirt\rr(z_0)z+\wirtq\rr(z_0)\overline{z}+\frac{1}{2}\wirt^2\rr(z_0)z^2+\wirt\wirtq\rr(z_0)z\overline{z}+\frac{1}{2}\wirtq^2\rr(z_0)\overline{z}^2+\Theta_2(z)
\end{equation}
for a function $\Theta_2:\CC\to\CC$ with $\lim_{\CC\setminus\setn{0}\ni z\to 0}\frac{\Theta_2(z)}{z^2}=0$.
\end{Lem}
\begin{proof}
\cref{app:taylor-1} follows from the definition of real differentiability \eqref{eq:real_diff} by using
\begin{equation*}
\frac{\p \rr}{\p x}(z_0) \RE(z) + \frac{\p \rr}{\p y} (z_0) \IM(z) = \frac{\partial \rr}{\partial x}(z_0)  \cdot \frac{1}{2}(z + \overline{z}) + \frac{\p \rr}{\p y}  \cdot \frac{1}{2 \ii} (z - \overline{z}) = \wirt \rr(z_0)z + \wirtq \rr (z_0) \overline{z}.
\end{equation*}
In order to prove \eqref{app:taylor-2} we use the second-order Taylor expansion of $\rr$ around $z_0$ which can be found for instance in \cite[Theorem~VII.5.11]{amann2008analysis} and obtain
\begin{equation*}
\rr(z + z_0) = \rr(z_0) + \frac{\p \rr}{\p x} (z_0) x + \frac{\p \rr}{\p y}(z_0) y + \frac{1}{2}\frac{\p^2 \rr}{\p x^2}(z_0)x^2 + \frac{\p^2 \rr}{\p x \p y}(z_0)xy + \frac{1}{2} \frac{\p^2 \rr}{\p y^2}(z_0) y^2 + \Theta_2(z)
\end{equation*}
where $\Theta_2 :\CC \to \CC$ satisfies $\lim_{\CC\setminus\setn{0}\ni z\to 0}\frac{\Theta_2(z)}{z^2}=0$.
Furthermore, we use the notation $x = \RE(z)$ and $y= \IM(z)$.
Letting $x = \frac{1}{2}(z + \overline{z})$, $y = \frac{1}{2\ii}(z-\overline{z})$ and using
\begin{equation}\label{eq:matrix}
\frac{1}{4}\pM{1& -2\ii&-1\\1&\phantom{-}0&\phantom{-}1\\1&\phantom{-}2\ii&-1} \pM{\frac{\p^2}{\p x^2} \\[4pt] \frac{\p^2}{\p x \p y} \\[4pt]\frac{\p^2}{\p y^2}} = \pM{\wirt^2 \\[4pt]\wirt \wirtq \\[4pt]\wirtq^2}
\end{equation}
yields the claim.
\end{proof}

The following Proposition is required in the proof of \cref{prop: approx}.
\begin{Prop} \label{app:helpproof}
Let $\rr \in C^2(\CC;\CC)$ be not $\RR$-affine.
Then there exists a point $z_0 \in \CC$ such that either $\wirt^2  \rr (z_0) \neq 0$, $\wirt \wirtq \rr (z_0) \neq 0$ or $\wirtq^2 \rr(z_0) \neq 0$.
\end{Prop}
\begin{proof}
Assume $\wirt^2 \rr \equiv \wirt \wirtq \rr \equiv \wirtq^2  \rr \equiv 0$.
From the fact that the matrix on the left-hand side in \eqref{eq:matrix} is invertible it follows $\frac{\p^2 \rr}{\p x^2} \equiv \frac{\p^2 \rr}{\p x \p y} \equiv \frac{\p^2 \rr}{\p y^2}  \equiv 0 $.
Since $\rr$ is $\RR$-affine if and only $\RE(\rr)$ and $\IM(\rr)$ are both $\RR$-affine, we may assume that $\rr$ is real-valued.
It is a well-known fact that a $C^1$-function with vanishing gradient is necessarily constant.
Applying this fact to $\frac{\p \rr}{\p x}$ and $\frac{\p \rr}{ \p y}$ separately shows
\begin{equation*}
\nabla \rr \equiv a
\end{equation*}
for a constant $a \in \RR^2$.
Let $f(z) \defeq z^\top a$ where $z \in \CC$ is treated as an element of $\RR^2$.
Then the gradient of $\rr - f$ vanishes identically and hence it holds $\rr - f \equiv b$ for a constant $b \in \RR$.
This yields
\begin{equation*}
\rr(z) = z^\top a + b \quad \text{for all } z \in \CC.
\end{equation*}
 But then $\rr$ is $\RR$-affine.
\end{proof}

\parindent 0pt
\textbf{Acknowledgements.} The authors thank Felix Voigtlaender for his helpful comments
and the anonymous reviewers for their feedback which helped improving the presentation of the material.
PG acknowledges support by the German Science Foundation (DFG) in
the context of the Emmy Noether junior research group VO 2594/1-1.

\footnotesize
\bibliographystyle{amsplain3}
\bibliography{references.bib}

\providecommand{\bysame}{\leavevmode\hbox to3em{\hrulefill}\thinspace}
\providecommand{\MR}{}
\providecommand{\MRhref}[2]{}
\providecommand{\href}[2]{#2}
\begin{thebibliography}{10}

\bibitem{amann2008analysis}
H.~Amann and J.~Escher, \emph{Analysis {II}}, Birkh\"{a}user Verlag, Basel,
  2008

\bibitem{arjovsky2016unitary}
M.~Arjovsky, A.~Shah, and Y.~Bengio, \emph{Unitary evolution recurrent neural
  networks}, International conference on machine learning, PMLR, 2016,
  pp.~1120--1128

\bibitem{barrachina2023comparison}
J.~A. Barrachina, C.~Ren, C.~Morisseau, G.~Vieillard, and J.-P. Ovarlez,
  \emph{Comparison between equivalent architectures of complex-valued and
  real-valued neural networks-application on polarimetric {SAR} image
  segmentation}, Journal of Signal Processing Systems \textbf{95} (2023),
  no.~1, pp.~57--66,
  \href{https://doi.org/10.1007/s11265-022-01793-0}{DOI:~10.1007/s11265-022-01793-0}

\bibitem{barron1993universal}
A.~R. Barron, \emph{Universal approximation bounds for superpositions of a
  sigmoidal function}, IEEE Transactions on Information theory \textbf{39}
  (1993), no.~3, pp.~930--945

\bibitem{Brandenberg2005}
R.~Brandenberg, \emph{Radii of regular polytopes}, Discrete Comput. Geom.
  \textbf{33} (2005), no.~1, pp.~43--55,
  \href{https://doi.org/10.1007/s00454-004-1127-1}{DOI:~10.1007/s00454-004-1127-1}
  \MR{2105749}

\bibitem{Bueno2021}
C.~Bueno, \emph{Universal approximation for neural nets on sets}, PhD thesis,
  University of California, Santa Barbara, 2021

\bibitem{Cai2022}
Y.~Cai, \emph{Achieve the minimum width of neural networks for universal
  approximation}, arXiv preprint, 2022,
  \href{http://arxiv.org/abs/2209.11395}{arXiv:~2209.11395}

\bibitem{CarageaLeMaPfVo2022}
A.~Caragea, D.~G. Lee, J.~Maly, G.~Pfander, and F.~Voigtlaender,
  \emph{Quantitative approximation results for complex-valued neural networks},
  SIAM J. Math. Data Sci. \textbf{4} (2022), no.~2, pp.~553--580,
  \href{https://doi.org/10.1137/21M1429540}{DOI:~10.1137/21M1429540}
  \MR{4414501}

\bibitem{cole2021analysis}
E.~Cole, J.~Cheng, J.~Pauly, and S.~Vasanawala, \emph{Analysis of deep
  complex-valued convolutional neural networks for mri reconstruction and
  phase-focused applications}, Magnetic Resonance in Medicine \textbf{86}
  (2021), no.~2, pp.~1093--1109,
  \href{https://doi.org/10.1002/mrm.28733}{DOI:~10.1002/mrm.28733}

\bibitem{Cybenko1989}
G.~Cybenko, \emph{Approximation by superpositions of a sigmoidal function},
  Math. Control Signals Systems \textbf{2} (1989), no.~4, pp.~303--314,
  \href{https://doi.org/10.1007/BF02551274}{DOI:~10.1007/BF02551274}
  \MR{1015670}

\bibitem{dugundji_topology_1987}
J.~Dugundji, \emph{Topology}, Allyn and Bacon, Inc., Boston, MA, 1966
  \MR{0193606}

\bibitem{folland_real_1999}
G.~B. Folland, \emph{Real analysis: {M}odern techniques and their
  applications}, 2nd ed ed., Pure and applied mathematics, Wiley, New York,
  1999

\bibitem{Geuchen2023}
P.~Geuchen and F.~Voigtlaender, \emph{Optimal approximation using
  complex-valued neural networks}, Advances in Neural Information Processing
  Systems (A.~Oh, T.~Naumann, A.~Globerson, K.~Saenko, M.~Hardt, and S.~Levine,
  eds.), vol.~36, Curran Associates, Inc., 2023, pp.~1681--1737.

\bibitem{GribonvalKuNiVo2022}
R.~Gribonval, G.~Kutyniok, M.~Nielsen, and F.~Voigtlaender, \emph{Approximation
  spaces of deep neural networks}, Constr. Approx. \textbf{55} (2022), no.~1,
  pp.~259--367,
  \href{https://doi.org/10.1007/s00365-021-09543-4}{DOI:~10.1007/s00365-021-09543-4}
  \MR{4376564}

\bibitem{IsmailovSa2023}
V.~E. Ismailov and E.~Savas, \emph{Measure theoretic results for approximation
  by neural networks with limited weights}, Numer. Funct. Anal. Optim.
  \textbf{38} (2017), no.~7, pp.~819--830,
  \href{https://doi.org/10.1080/01630563.2016.1254654}{DOI:~10.1080/01630563.2016.1254654}
  \MR{3654363}

\bibitem{kaup_holomorphic_1983}
L.~Kaup and B.~Kaup, \emph{Holomorphic functions of several variables: An
  introduction to the fundamental theory}, De Gruyter Studies in Mathematics,
  vol.~3, Walter de Gruyter \& Co., Berlin, 1983,
  \href{https://doi.org/10.1515/9783110838350}{DOI:~10.1515/9783110838350}
  \MR{716497}

\bibitem{KidgerLy2020}
P.~Kidger and T.~Lyons, \emph{Universal approximation with deep narrow
  networks}, Proceedings of {T}hirty {T}hird {C}onference on {L}earning
  {T}heory (J.~Abernethy and S.~Agarwal, eds.), vol. 125, 2020, pp.~2306--2327,
  \href{https://proceedings.mlr.press/v125/kidger20a.html}{https://proceedings.mlr.press/v125/kidger20a.html}

\bibitem{kratsios2022universal}
A.~Kratsios and L.~Papon, \emph{Universal approximation theorems for
  differentiable geometric deep learning}, Journal of Machine Learning Research
  \textbf{23} (2022), no.~196, pp.~1--73

\bibitem{kustner2020cinenet}
T.~K{\"u}stner, N.~Fuin, K.~Hammernik, A.~Bustin, H.~Qi, R.~Hajhosseiny, P.~G.
  Masci, R.~Neji, D.~Rueckert, R.~M. Botnar, and C.~Prieto, \emph{{CINENet}:
  deep learning-based 3d cardiac {CINE} {MRI} reconstruction with multi-coil
  complex-valued {4D} spatio-temporal convolutions}, Scientific Reports
  \textbf{10} (2020), pp.~article no. 13710, 13 pp.,
  \href{https://doi.org/10.1038/s41598-020-70551-8}{DOI:~10.1038/s41598-020-70551-8}

\bibitem{lee2022complex}
C.~Y. Lee, H~Hasegawa, and S.~C. Gao, \emph{Complex-valued neural networks: A
  comprehensive survey}, IEEE/CAA Journal of Automatica Sinica \textbf{9}
  (2022), no.~8, pp.~1406--1426,
  \href{https://doi.org/10.1109/JAS.2022.105743}{DOI:~10.1109/JAS.2022.105743}

\bibitem{lei2023fully}
Z.~Lei, S.~Gao, H.~Hasegawa, Z.~Zhang, M.~Zhou, and K.~Sedraoui, \emph{Fully
  complex-valued gated recurrent neural network for ultrasound imaging}, IEEE
  Transactions on Neural Networks and Learning Systems \textbf{35} (2024),
  no.~10, pp.~14918--14931,
  \href{https://doi.org/10.1109/TNNLS.2023.3282231}{DOI:~10.1109/TNNLS.2023.3282231}

\bibitem{Pinkus1999}
M.~Leshno, V.~Ya. Lin, A.~Pinkus, and S.~Schocken, \emph{Multilayer feedforward
  networks with a nonpolynomial activation function can approximate any
  function}, Neural Netw. \textbf{6} (1993), no.~6, pp.~861--867,
  \href{https://doi.org/10.1016/S0893-6080(05)80131-5}{DOI:~10.1016/S0893-6080(05)80131-5}

\bibitem{LiWa2011}
X.~Li and G.~Wang, \emph{Universal approximation of polygonal fuzzy neural
  networks in sense of {$K$}-integral norms}, Sci. China Inf. Sci. \textbf{54}
  (2011), no.~11, pp.~2307--2323,
  \href{https://doi.org/10.1007/s11432-011-4364-y}{DOI:~10.1007/s11432-011-4364-y}
  \MR{2846532}

\bibitem{LuPuWaHuWa2017}
Zh. Lu, H.~Pu, F.~Wang, Zh. Hu, and L.~Wang, \emph{The expressive power of
  neural networks: {A} view from the width}, Advances in Neural Information
  Processing Systems (I.~Guyon, U.~Von Luxburg, S.~Bengio, H.~Wallach,
  R.~Fergus, S.~Vishwanathan, and R.~Garnett, eds.), Curran Associates, Inc.,
  2017, pp.~1--9.

\bibitem{MerkhMo2019}
T.~Merkh and G.~Mont\'ufar, \emph{Stochastic feedforward neural networks:
  universal approximation}, Mathematical aspects of deep learning, Cambridge
  Univ. Press, Cambridge, 2023, pp.~267--314 \MR{4505888}

\bibitem{mhaskar1996neural}
H.~N Mhaskar, \emph{Neural networks for optimal approximation of smooth and
  analytic functions}, Neural computation \textbf{8} (1996), no.~1,
  pp.~164--177

\bibitem{MurataSo2017}
N.~Murata and S.~Sonoda, \emph{Neural network with unbounded activation
  functions is universal approximator}, Appl. Comput. Harmon. Anal. \textbf{43}
  (2017), no.~2, pp.~233--268,
  \href{https://doi.org/10.1016/j.acha.2015.12.005}{DOI:~10.1016/j.acha.2015.12.005}
  \MR{3668038}

\bibitem{Park2022}
J.~Park and S.~Wojtowytsch, \emph{Qualitative neural network approximation over
  {$\mathbb{R}$} and {$\mathbb{C}$}: {E}lementary proofs for analytic and
  polynomial activation}, arXiv preprint, 2022,
  \href{http://arxiv.org/abs/2203.13410}{arXiv:~2203.13410}

\bibitem{ParkYuLeSh2020}
S.~Park, C.~Yun, J.~Lee, and J.~Shin, \emph{Minimum width for universal
  approximation}, arXiv preprint, 2020,
  \href{http://arxiv.org/abs/2006.08859}{arXiv:~2006.08859}

\bibitem{petersen2018optimal}
P.~Petersen and F.~Voigtlaender, \emph{Optimal approximation of piecewise
  smooth functions using deep {ReLU} neural networks}, Neural Networks
  \textbf{108} (2018), pp.~296--330

\bibitem{PetersenVo2020}
\bysame, \emph{Equivalence of approximation by convolutional neural networks
  and fully-connected networks}, Proc. Amer. Math. Soc. \textbf{148} (2020),
  no.~4, pp.~1567--1581,
  \href{https://doi.org/10.1090/proc/14789}{DOI:~10.1090/proc/14789}
  \MR{4069195}

\bibitem{qu2023entanglement}
Y.-D. Qu, R.-Q. Zhang, S.-Q. Shen, J.~Yu, and M.~Li, \emph{Entanglement
  detection with complex-valued neural networks}, International Journal of
  Theoretical Physics \textbf{62} (2023), no.~9, pp.~article no. 206, 15 pp.,
  \href{https://doi.org/10.1007/s10773-023-05460-3}{DOI:~10.1007/s10773-023-05460-3}

\bibitem{ren2023new}
Y.~Ren, W.~Jiang, and Y.~Liu, \emph{A new architecture of a complex-valued
  convolutional neural network for {PolSAR} image classification}, Remote
  Sensing \textbf{15} (2023), no.~19, pp.~article no. 4801, 27 pp.,
  \href{https://doi.org/10.3390/rs15194801}{DOI:~10.3390/rs15194801}

\bibitem{Rudin1976}
W.~Rudin, \emph{Principles of mathematical analysis}, third ed., International
  Series in Pure and Applied Mathematics, McGraw-Hill Book Co., New
  York-Auckland-D\"{u}sseldorf, 1976 \MR{0385023}

\bibitem{Rudin1987}
\bysame, \emph{Real and complex analysis}, third ed., McGraw-Hill Book Co., New
  York, 1987 \MR{924157}

\bibitem{virtue2017better}
P.~Virtue, S.~X. Yu, and M.~Lustig, \emph{Better than real: {C}omplex-valued
  neural nets for {MRI} fingerprinting}, 2017 IEEE International Conference on
  Image Processing (ICIP), IEEE, 2017, pp.~3953--3957

\bibitem{Voigtlaender2022}
F.~Voigtlaender, \emph{The universal approximation theorem for complex-valued
  neural networks}, Appl. Comput. Harmon. Anal. \textbf{64} (2023), pp.~33--61,
  \href{https://doi.org/10.1016/j.acha.2022.12.002}{DOI:~10.1016/j.acha.2022.12.002}
  \MR{4530638}

\bibitem{yarotsky2018optimal}
D.~Yarotsky, \emph{Optimal approximation of continuous functions by very deep
  {ReLU} networks}, Conference on learning theory, PMLR, 2018, pp.~639--649

\bibitem{Yarotsky2022}
\bysame, \emph{Universal approximations of invariant maps by neural networks},
  Constr. Approx. \textbf{55} (2022), no.~1, pp.~407--474,
  \href{https://doi.org/10.1007/s00365-021-09546-1}{DOI:~10.1007/s00365-021-09546-1}
  \MR{4376566}

\bibitem{zhang2021optical}
H.~Zhang, M.~Gu, X.~D. Jiang, J.~Thompson, H.~Cai, S.~Paesani, R.~Santagati,
  A.~Laing, Y.~Zhang, M.~H. Yung, Y.~Z. Shi, F.~K. Muhammad, G.~Q. Lo, X.~S.
  Luo, B.~Dong, D.~L. Kwong, L.~C. Kwek, and A.~Q. Liu, \emph{An optical neural
  chip for implementing complex-valued neural network}, Nature Communications
  \textbf{12} (2021), pp.~article no. 457, 11 pp.,
  \href{https://doi.org/10.1038/s41467-020-20719-7}{DOI:~10.1038/s41467-020-20719-7}

\bibitem{zhang2017complex}
Z.~Zhang, H.~Wang, F.~Xu, and Y.-Q. Jin, \emph{Complex-valued convolutional
  neural network and its application in polarimetric {SAR} image
  classification}, IEEE Transactions on Geoscience and Remote Sensing
  \textbf{55} (2017), no.~12, pp.~7177--7188,
  \href{https://doi.org/10.1109/TGRS.2017.2743222}{DOI:~10.1109/TGRS.2017.2743222}

\bibitem{Zhou2020}
D.-X. Zhou, \emph{Universality of deep convolutional neural networks}, Appl.
  Comput. Harmon. Anal. \textbf{48} (2020), no.~2, pp.~787--794,
  \href{https://doi.org/10.1016/j.acha.2019.06.004}{DOI:~10.1016/j.acha.2019.06.004}
  \MR{4047545}

\end{thebibliography}
\end{document}